\documentclass[11pt]{amsart}
\usepackage{amscd,amssymb}
\usepackage{amsthm,amsmath,amssymb}
\usepackage[matrix,arrow]{xy}
\usepackage{enumerate}
\usepackage{tikz}

\newtheorem{theorem}[equation]{Theorem}

\newtheorem{lemma}[equation]{Lemma}
\newtheorem{corollary}[equation]{Corollary}
\newtheorem{conjecture}[equation]{Conjecture}
\newtheorem{question}[equation]{Question}

\theoremstyle{definition}

\newtheorem{definition}[equation]{Definition}

\theoremstyle{remark}
\newtheorem{remark}[equation]{Remark}

\makeatletter\@addtoreset{equation}{section} \makeatother

\newcommand{\PP}{\mathbb{P}}
\newcommand{\Q}{\mathbb{Q}}

\DeclareMathOperator{\mult}{mult}

\author[Ahmadinezhad \& Cheltsov \& Schicho]{Hamid Ahmadinezhad, Ivan Cheltsov and Josef Schicho}

\title[On a conjecture of Tian]{On a conjecture of Tian}

\address{Department of Mathematical Sciences, Loughborough University, LE11 3TU, UK}
\email{h.ahmadinezhad@lboro.ac.uk}

\address{University of Edinburgh, Department of Mathematics, Mayfield Rd., Edinburgh EH9 3JZ, UK}
\email{I.Cheltsov@ed.ac.uk}

\address{RISC, Johannes Kepler University, Linz, Schloss Hagenberg, 4232 Hagenberg, Austria}
\email{josef.schicho@risc.jku.at}



\dedicatory{``A tragedy of mathematics is a beautiful conjecture ruined by an ugly fact.''}

\begin{document}

\begin{abstract}
We study Tian's $\alpha$-invariant in comparison with the $\alpha_1$-invariant for pairs $(S_d,H)$ consisting of
a smooth surface $S_d$ of degree $d$ in the projective three-dimensional space and a hyperplane section $H$.
A conjecture of Tian asserts that $\alpha(S_d,H)=\alpha_1(S_d,H)$.
We show that this is indeed true for $d=4$ (the result is well known for $d\leqslant 3$),
and we show that $\alpha(S_d,H)<\alpha_1(S_d,H)$ for $d\geqslant 8$ provided that $S_d$ is general enough. We also construct examples of $S_d$, for $d=6$ and $d=7$, for which Tian's conjecture fails. We provide a candidate counterexample for~$S_5$.
\end{abstract}

\sloppy

\maketitle

\setcounter{tocdepth}{1}

\section{Introduction}
\label{section:into}

In order to prove the existence of a K\"ahler-Einstein metric, known as the Calabi problem, on a smooth Fano variety, in \cite{Tian97} Gang Tian introduced a quantity, known as the $\alpha$-invariant, that measures how singular pluri-anticanonical divisors on the Fano variety can be. There, he proved that a smooth Fano variety of dimension $m$ admits a K\"ahler-Einstein metric provided that its $\alpha$-invariant is bigger that $\frac{m}{m+1}$.

Despite the fact that the Calabi problem for smooth Fano varieties has been solved  (see \cite{CDS,Eyss,Rubinstein,Tian-CPAM})  this result of Tian is often the only way to prove the existence of the K\"ahler-Einstein metric for a given Fano.

In fact, the $\alpha$-invariant turned out to have important applications in birational geometry as well; see for example \cite{Vanya}.
Later, Tian generalised this invariant for arbitrary polarised pairs $(X,L)$, where $X$ is a smooth variety and $L$ is an ample Cartier divisor on it.
For the pair $(X,L)$, it can be defined as
$$
\alpha\big(X,L\big)=\mathrm{sup}\left\{\lambda\in\mathbb{Q}\ \left|%
\aligned
&\text{the log pair}\ \left(X, \lambda D\right)\ \text{is log canonical}\\
&\text{for every effective $\mathbb{Q}$-divisor}\ D\sim_{\mathbb{Q}} L
\endaligned\right.\right\}\in\mathbb{R}_{>0}.%
$$
This definition coincides with Tian's original definition in \cite{Tian97,Tian2012} by \cite[Theorem~A.3]{ChSh}.

The number $\alpha(X,L)$ is often hard to compute but, in good situations, can be approximated by numbers
that are much easier to control (see, for example, \cite[Proposition~2.2]{ChPaWo14}).
For instance, if the linear system $|nL|$ is not empty,
Tian defined the $n$-th $\alpha$-invariant of the pair $(X,L)$ as
$$
\alpha_n\big(X,L\big)=\mathrm{sup}\Bigg\{\lambda\in\mathbb{Q}\ \Bigg|\ \text{the pair}\ \Bigg(X, \frac{\lambda}{n} D\Bigg)\ \text{is log canonical for every}\ D\in|nL|\Bigg\}\in\mathbb{Q}_{>0}.%
$$
If the linear system $|nL|$ is empty, one can simply put $\alpha_n(X,L)=+\infty$.
Then $\alpha(X,L)\leqslant\alpha_n(X,L)$ and
$$
\alpha\big(X,L\big)=\inf_{n\geqslant 1}\Big\{\alpha_n\big(X,L\big)\Big\}.
$$

Then, Tian posed the following conjecture.

\begin{conjecture}[{\cite[Conjecture~5.4]{Tian2012}}]
\label{conjecture:Tian}
Suppose that $L$ is very ample and defines a projectively normal embedding under its associated morphism,
i.e., the graded algebra
$$
\bigoplus_{i\geqslant 0}H^0\Big(X,\mathcal{O}_X\big(iL\big)\Big)
$$
is generated by elements in $H^0(X,\mathcal{O}_X(L))$. Then $\alpha(X,L)=\alpha_1(X,L)$.
\end{conjecture}

Note that the very ampleness of the divisor $L$ does not always imply that the associated morphism gives a projectively normal embedding.
However, in many cases this is true, for example when $X$ is a hypersurface and $L$ is a hyperplane section, which includes all varieties we study in this article.
Note also that \cite[Conjecture~5.4]{Tian2012} is stated in terms of the more delicate invariants $\alpha_{n,k}(X,L)$,
which are defined in analytic language (for their explicit definitions see \cite[\S~5]{Tian2012}).
Arguing as in the proof of \cite[Theorem~A.3]{ChSh}, one can show that
$$
\alpha_n\big(X,L\big)=\alpha_{n,1}\big(X,L\big),
$$
so that  Conjecture~\ref{conjecture:Tian} is a special case of Tian's more general \cite[Conjecture~5.4]{Tian2012}.

The purpose of this paper is to study Conjecture~\ref{conjecture:Tian} for smooth surfaces in $\mathbb{P}^3$.
Namely, let $S_d$ be a smooth surface in $\mathbb{P}^3$ of degree $d\geqslant 1$, and let $H$ be its hyperplane section.
Then the pair $(S_d,H)$ satisfies all hypotheses of Conjecture~\ref{conjecture:Tian}.
Moreover, if $d=1$ or $d=2$, then 
$$
\alpha(S_d,H)=\alpha_1(S_d,H)=1.
$$
Indeed, in these cases $S_d$ is toric, so that the required equalities follows from \cite[Lemma~5.1]{ChSh}.
Furthermore, if $d=3$, then $\alpha(S_d,H)=\alpha_1(S_d,H)$ by \cite[Theorem~1.7]{Ch08}.
In Section~\ref{section:quartic}, we prove

\begin{theorem}
\label{theorem:quartic}
Let $S_4$ be a smooth quartic surface in $\mathbb{P}^3$. Then $\alpha(S_4,H)=\alpha_1(S_4,H)$.
\end{theorem}

Hence, Conjecture~\ref{conjecture:Tian} holds for the pair $(S_d,H)$ provided that $d\leqslant 4$.
In particular, this gives an easy way to compute all possible values of $\alpha(S_d,H)$ for $d=4$, 
because the number $\alpha_1(S_d,H)$ is easy to compute.
However, Conjecture~\ref{conjecture:Tian} fails for \emph{general} surfaces of large degree in $\mathbb{P}^3$.
This follows from

\begin{theorem}
\label{theorem:general-surface}
Let $S_d$ be a general surface in $\mathbb{P}^3$ of degree $d\geqslant 8$.
Then $\alpha(S_d,H)<\alpha_1(S_d,H)$.
\end{theorem}

This result shows that it is hard to compute $\alpha(S_d,H)$ for $d\gg 0$. 
In fact, we do not know what the exact value of $\alpha(S_d,H)$ is when $d\geqslant 5$ and the surface $S_d$ is general.
One the other hand, we prove that $\alpha_1(S_d,H)=\frac{3}{4}$ for these hypersurfaces (see Lemmas~\ref{lemma:3-4} and \ref{lemma:3-4-general}).

We prove Theorem~\ref{theorem:general-surface} in Section~\ref{section:general}.
In Section~\ref{section:small-degree}, we show that Conjecture~\ref{conjecture:Tian} also fails for
\emph{some} smooth sextic and septic surfaces in $\mathbb{P}^3$.
We believe that it fails for \emph{some} smooth quintic surfaces  as well.
Unfortunately, we are unable to verify this claim at this stage, due to enormous computations required in our method (see Remark~\ref{computer}).

By \cite[Theorem~1.7]{Ch08}, Conjecture~\ref{conjecture:Tian} holds for all smooth del Pezzo surfaces,
i.e.\ smooth Fano varieties of dimension two, polarized by their anticanonical divisors.
Surfaces considered in Theorem~\ref{theorem:quartic} and \ref{theorem:general-surface} have non-negative Kodaira dimension,
so that, in particular, they are not del Pezzo surfaces.
Unfortunately, we do not know whether Conjecture~\ref{conjecture:Tian} holds for smooth del Pezzo surfaces polarised by arbitrary ample divisors.
Thus, we conclude by posing

\begin{question}[Rubinstein]
\label{question:del-Pezzo}
Let $S$ be a smooth del Pezzo surface. Is it true that
$$
\alpha\big(S,A\big)=\alpha_1\big(S,A\big)
$$
for every ample divisor $A\in\mathrm{Pic}(S)$.
\end{question}

All varieties are assumed to be algebraic, projective and defined over $\mathbb{C}$.

\section{Singularities of pairs}
\label{section:background}

In this section we present local results about effective $\mathbb{Q}$-divisors on smooth surfaces.
Almost all these results can be found in \cite[\S~6]{CoKoSm} in much more general forms.

Let $S$ be a~smooth surface, let $D$ be an effective non-zero
$\mathbb{Q}$-divisor on the surface $S$, and let $P$ be a point in
the surface $S$. Put $D=\sum_{i=1}^{r}a_iC_i$, where each $C_i$ is
an irreducible curve on $S$, and each $a_i$ is a non-negative
rational number. We assume here that all curves $C_1,\ldots,C_r$ are different.
We call $(S,D)$ a \emph{log pair}.

Let $\pi\colon\widetilde{S}\to S$ be a
birational morphism such that $\widetilde{S}$ is also smooth.
Then $\pi$ is a composition of $n$ blow ups of smooth points.
For each $C_i$, denote by $\widetilde{C}_i$ its proper transform on the surface $\widetilde{S}$.
Let $F_1,\ldots, F_n$ be $\pi$-exceptional curves.
Then
$$
K_{\widetilde{S}}+\sum_{i=1}^{r}a_i\widetilde{C}_i+\sum_{j=1}^{n}b_jF_j\sim_{\mathbb{Q}}\pi^{*}\big(K_{S}+D\big)
$$
for some rational numbers $b_1,\ldots,b_n$. Suppose, in addition,
that $\sum_{i=1}^r\widetilde{C}_i+\sum_{j=1}^n F_j$ is a divisor with
simple normal crossings.

\begin{definition}
\label{definition:lct}  The log pair $(S,D)$ is said to be \emph{log canonical} at the point $P$ if
the following two conditions are satisfied:
\begin{itemize}
\item $a_i\leqslant 1$ for every $C_i$ such that $P\in C_i$,

\item $b_j\leqslant 1$ for every $F_j$ such that $\pi(F_j)=P$.
\end{itemize}
\end{definition}

This definition is independent on the choice of birational morphism $\pi\colon\widetilde{S}\to S$
provided that the surface $\widetilde{S}$ is smooth and $\sum_{i=1}^r\widetilde{C}_i+\sum_{j=1}^n F_j$ is a divisor with
simple normal crossings. The log pair $(S,D)$ is said to be \emph{log canonical} if it is log
canonical at every point of $S$.

\begin{remark}
\label{remark:convexity} Let $R$ be any effective
$\mathbb{Q}$-divisor on $S$ such that $R\sim_{\mathbb{Q}} D$ and
$R\ne D$. Put
$$
D_{\epsilon}=(1+\epsilon) D-\epsilon R
$$
for some rational number $\epsilon\geqslant 0$. Then $D_{\epsilon}\sim_{\mathbb{Q}} D$.
Moreover, there exists the greatest rational number
$\epsilon_0\geqslant 0$ such that the divisor $D_{\epsilon_0}$ is effective.
Then $\mathrm{Supp}(D_{\epsilon_0})$ does not contain at least one
irreducible component of $\mathrm{Supp}(R)$. Moreover, if $(S,D)$
is not log canonical at $P$, and $(S,R)$ is log canonical at $P$,
then $(S,D_{\epsilon_0})$ is not log canonical at $P$ by
Definition~\ref{definition:lct}, because
$$
D=\frac{1}{1+\epsilon_0}D_{\epsilon_0}+\frac{\epsilon_0}{1+\epsilon_0}R.
$$
\end{remark}

The following result is well-known and is very easy to prove.

\begin{lemma}[{\cite[Exercise~6.18]{CoKoSm}}]
\label{lemma:Skoda} If $(S,D)$ is not log canonical at $P$, then $\mathrm{mult}_{P}(D)>1$.
\end{lemma}

Let $\pi_1\colon S_1\to S$ be a~blow up of the point $P$, and let
$E_1$ be the $\pi_1$-exceptional curve. Denote by $D^1$ the~proper
transform of the divisor $D$ on the surface $S_1$ via $\pi_1$.
Then
$$
K_{S_1}+D^1+\Big(\mathrm{mult}_{P}(D)-1\Big)E_1\sim_{\mathbb{Q}}\pi_1^{*}\big(K_{S}+D\big).
$$

\begin{remark}
\label{remark:log-pull-back} The log pair $(S,D)$ is log canonical
at $P$ if and only if $(S_1, D^1+(\mathrm{mult}_{P}(D)-1)E_1)$ is
log canonical at every point of the curve $E_1$.
\end{remark}

\begin{corollary}
\label{corollary:lc-mult} If $\mathrm{mult}_{P}(D)>2$, then $(S,D)$ is not log canonical at $P$.
\end{corollary}

We can measure how far the pair $(S,D)$ is from being log canonical at $P$ by the positive rational number
$$
\mathrm{lct}_P\big(S,D\big)=\sup\Big\{\lambda\in\Q\mid\text{the log pair }(S,\lambda D)\text{ is log canonical at } P\Big\}.
$$
This number has been introduced by Shokurov and is called the \emph{log canonical threshold} of the pair $(S,D)$ at the point $P\in S$.
The log canonical threshold of the pair $(S,D)$ is defined as
$$
\mathrm{lct}\big(S,D\big)=\inf_{O\in S}\Big\{\mathrm{lct}_O(S,D)\Big\}.
$$
By Lemma~\ref{lemma:Skoda} and Corollary~\ref{corollary:lc-mult}, we have

\begin{equation}
\label{equation:lct}
\frac{2}{\mathrm{mult}_{P}(D)}\geqslant\mathrm{lct}_P\big(S,D\big)\geqslant\frac{1}{\mathrm{mult}_{P}(D)}.
\end{equation}

The following theorem is a very special case of a much more general
result known as \emph{Inversion of Adjunction} (see, for example, \cite[Theorem~6.29]{CoKoSm}).

\begin{theorem}[{\cite[Exercise~6.31]{CoKoSm}, \cite[Theorem~7]{Ch13}}]
\label{theorem:adjunction}
Suppose that $r\geqslant 2$. Put $\Delta=\sum_{i=2}^{r}a_iC_i$.
Suppose that $C_1$ is smooth at $P$, $a_1\leqslant 1$, and the log pair $(S,D)$ is not log canonical at $P$.
Then  $\mathrm{mult}_{P}(C_1\cdot\Delta)>1$.
\end{theorem}

This theorem implies

\begin{lemma}
\label{lemma:log-pull-back} Suppose that $(S,D)$ is not log canonical at $P$, and $\mathrm{mult}_{P}(D)\leqslant 2$.
Then there exists a unique point in $E_1$ such that $(S_1, D^1+(\mathrm{mult}_{P}(D)-1)E_1)$ is not log canonical at it.
\end{lemma}

\begin{proof}
If $\mathrm{mult}_{P}(D)\leqslant 2$ and $(S_1, D^1+(\mathrm{mult}_{P}(D)-1)E_1)$ is not log canonical at two distinct
points $P_1$ and $\widetilde{P}_1$ of the curve $E_1$, then
$$
2\geqslant\mathrm{mult}_{P}\big(D\big)=D^1\cdot E_1\geqslant\mathrm{mult}_{P_1}\Big(D^1\cdot E_1\Big)+\mathrm{mult}_{\widetilde{P}_1}\Big(D^1\cdot E_1\Big)>2%
$$
by Theorem~\ref{theorem:adjunction}. By
Remark~\ref{remark:log-pull-back}, this proves the
assertion.
\end{proof}

A crucial role in the proof of Theorems~\ref{theorem:quartic} is played by

\begin{theorem}[{\cite[Theorem~13]{Ch13}}]
\label{theorem:Trento} Suppose that $r\geqslant 3$. Put $\Delta=\sum_{i=3}^{r}a_iC_i$.
Suppose that the curves $C_1$ and $C_2$ are smooth at $P$ and intersect each other transversally at $P$,
the log pair $(S,D)$ is not log canonical at $P$, and $\mathrm{mult}_{P}(\Delta)\leqslant 1$.
Then either
$$
\mathrm{mult}_{P}\Big(C_1\cdot\Delta\Big)>2\big(1-a_{2}\big)
$$
or
$$
\mathrm{mult}_{P}\Big(C_1\cdot\Delta\Big)>2\big(1-a_{1}\big)
$$
(or both).
\end{theorem}

Recall that $\pi$ is a composition of $n$ blow ups of smooth points.
We encourage the reader to prove both Theorems~\ref{theorem:adjunction} and \ref{theorem:Trento} using induction on $n$.

\section{Smooth surfaces in $\mathbb{P}^3$}
\label{section:surfaces}

In this section we collect global results about smooth surfaces in $\mathbb{P}^3$.
These results will be used in the proof of Theorems~\ref{theorem:quartic} and \ref{theorem:general-surface}.

Let $S_d$ be a smooth surface in $\mathbb{P}^3$ of degree $d$. Denote by $H$ its hyperplane section.
Then
$$
1\geqslant\alpha\big(S_d, H\big)\geqslant\frac{1}{d}.
$$
by Lemma~\ref{lemma:Skoda}.
These bounds are not optimal for $d\geqslant 2$.
In fact, if $d\geqslant 2$, then $\alpha(S_d, H)\geqslant\frac{2}{d}$.
Moreover, $\alpha(S_d, H)=\frac{2}{d}$ if and only if $S_d$ contains a so-called \emph{star point},
i.e., a point that is an intersection of $d$ lines contained in $S_d$.
This follows from \cite[Corollary~1.27]{Ch14}.
A slightly better upper bound for $\alpha(S_d, H)$ follows from

\begin{lemma}
\label{lemma:3-4}
Suppose that $d\geqslant 3$. Then $\alpha_1(S_d,H)\leqslant\frac{3}{4}$.
\end{lemma}

\begin{proof}
Let us first consider the case $d=3$.
Then $S_3$ is a smooth cubic surface in $\mathbb{P}^3$.
It is well-known that $S_3$ contains $27$ lines.
Taking hyperplane sections of the cubic surface $S_3$ passing through one of these lines $L_1$,
we see that either there exists a conic $C$ in $S_3$ such that 
$$
L_1+C\sim H
$$
and $L_1$ is tangent to $C$, or $S_3$ contains two more lines $L_2$ and $L_3$ such that 
$$
L_1+L_2+L_3\sim H
$$
and all three lines $L_1$, $L_2$ and $L_3$ intersect in a single point.
In the former case, one has $\alpha_1(S_d,H)\leqslant\frac{3}{4}$ by definition of $\alpha_1(S_d,H)$.
Similarly, in the later case, one has $\alpha_1(S_d,H)\leqslant\frac{2}{3}$.

We proved the required assertion in the case $d=3$.
Now let us prove it for $d=4$. The proof is similar for higher degrees.

Let $\mathcal{X}\cong\PP^{34}$ be the variety of all quartics in four variables,
and suppose $\mathcal{Y}$ is the variety of all complete flag varieties in $\mathbb{P}^3$, hence $\mathcal{Y}$ is a projective variety of dimension $6$.
Consider the incidence variety $\mathcal{Z}\subset \mathcal{X}\times \mathcal{Y}$ consisting of all pairs $(X, Y)$,
where  $Y=(P, L, E)$, such that $X\cap E$ has an $\mathbb{A}_3$, or worse, singularity at $P$ with tangent $L$.
We claim that the fibres of the second projection are linear subspaces of codimension $6$.
To show this, we choose a coordinate system such that $P$, $L$ and $E$ are, respectively, defined by $x=y=z=0$, $x=y=0$ and $x=0$.
Then the fibre of $\mathcal{Y}$ is the set of quartics such that the coefficients of the monomials
$$
yzw^2, yw^3, z^3w, z^2w^2, zw^3, w^4
$$
are equal to zero.

Therefore it follows that $\mathcal{Z}$ is irreducible and has dimension $34+6-6=34$.
In order to complete the proof, we need to show that the first projection is surjective.
Since it is a projective map, the image $\mathcal{W}\subset \mathcal{X}$ is closed.
We claim that there exists a point $X\in W$ with finite fibre.
Then the generic fibre is finite and $\dim(\mathcal{W})=\dim(\mathcal{Z})=34$.

A quartic surface corresponds to a point $X_0\in\mathcal{W}$ with finite fiber if it is nonsingular and the intersections
with its tangent planes do not have triple points; equivalently, the rank of the hessian of the
equation of the surface never drops to $2$. An example of such a surface is given by the equation
$$
x^4+y^4+z^4+w^4+(x^2+y^2+z^2+w^2)^2=0.
$$
\end{proof}

Arguing as in the proof of \cite[Proposition~2.1]{ChPaWo14}, we get

\begin{lemma}
\label{lemma:3-4-general}
Suppose that $S_d$ is a general surface in $\mathbb{P}^3$ of degree $d$.
Then $\alpha_1(S_d,H)\geqslant\frac{3}{4}$.
\end{lemma}

\begin{proof}
Similar as in the proof of Lemma~\ref{lemma:3-4}, we define $\mathcal{X}\cong\PP^{{d+3\choose 3}-1}$,
$\mathcal{Y}$ the variety of all complete flag varieties,
and $\mathcal{Z}\subset \mathcal{X}\times\mathcal{Y}$ the incidence consisting of all pairs $(X, Y)$,
where $Y=(P, L, E)$, such that $X\cap E$ has an $\mathbb{A}_4$, or worse, singularity at $P$ with tangent $L$.
Now the fibers of the second projection have codimension $7$ (defined by $6$ linear and one quadratic equation).
Since $\dim(\mathcal{Y})=6$, it follows that
$\dim(\mathcal{Z})<\dim(\mathcal{X})$, hence the first projection cannot be surjective and the generic surface has no corresponding
point in $\mathcal{Z}$. This shows that its hyperplane sections have only singularities of type $\mathbb{A}_1$, $\mathbb{A}_2$, and $\mathbb{A}_3$.
\end{proof}

The following result is due to Pukhlikov.

\begin{lemma}
\label{lemma:Pukhlikov}
Let $D$ be an effective $\mathbb{Q}$-divisor on $S_d$ such that $D\sim_{\mathbb{Q}} H$,
and let $P$ be a point in the surface $S_d$.
Put $D=\sum_{i=1}^{r}a_iC_i$, where each $C_i$ is an irreducible curve,
and each $a_i$ is a non-negative rational number.
Then each $a_i$ does not exceed~$1$.
\end{lemma}

\begin{proof}
Let $X$ be a cone over the curve $C_i$ whose vertex is a sufficiently general point in $\mathbb{P}^3$.
Then
$$
X\cap S=C_i+\widehat{C}_i,
$$
where $\widehat{C}_i$ is an irreducible curve of degree $(d-1)\mathrm{deg}(C_i)$.
Moreover, $\widehat{C}_i$ is not contained in the support of the divisor $D$.
Furthermore, the intersection $C_i\cap\widehat{C}_i$ consists of $\mathrm{deg}(\widehat{C}_i)$ different points,
because the surface $S_d$ is smooth. Thus, we have
$$
\mathrm{deg}\big(\widehat{C}_i\big)=D\cdot \widehat{C}_i\geqslant a_iC_i\cdot \widehat{C}_i\geqslant a_i\mathrm{deg}\big(\widehat{C}_i\big),
$$
which implies that $a_i\leqslant 1$.
\end{proof}

For an alternative proof of Pukhlikov's lemma, see the proof of \cite[Lemma~5.36]{CoKoSm}.

\section{Quartic surfaces}
\label{section:quartic}

In this section, we prove Theorem~\ref{theorem:quartic}.
Let $S_4$ be a smooth quartic surface in $\mathbb{P}^3$. Denote by $H$ its hyperplane section.
By definition, one has $\alpha(S_4,H)\leqslant\alpha_1(S_4,H)$.
We must show that $\alpha(S_4,H)=\alpha_1(S_4,H)$.
Suppose that $\alpha(S_4,H)<\alpha_1(S_4,H)$.
Let us seek for a contradiction.

Since $\alpha(S_4,H)<\alpha_1(S_4,H)$, there exists an effective $\Q$-divisor $D$ such that
$D\sim_{\mathbb{Q}} H$ and $(S_4,\lambda D)$ is not log canonical for some $\lambda<\alpha_1(S_4,H)$.
Since $\alpha_1(S_4,H)\leqslant\frac{3}{4}$ by Lemma~\ref{lemma:3-4}, we have

\begin{equation}
\label{equation:lambda-3-4}
\lambda<\frac{3}{4}.
\end{equation}

By Lemma~\ref{lemma:Pukhlikov}, the log pair $(S_4,\lambda D)$ is log canonical outside of finitely many points.
Let $P$ be one of these points at which $(S_4,\lambda D)$ is not log canonical.
Consider the quartic curve $T_P$ that is cut out on $S_4$ by the hyperplane in $\PP^3$ that is tangent to $S_4$ at the point $P$.
Then $T_P$ is a reduced plane quartic curve Lemma~\ref{lemma:Pukhlikov}.
It is singular at the point $P$ by construction.

\begin{lemma}
\label{lemma:curve-C}
The curve $T_P$ contains all lines in $S_4$ that passes through $P$.
\end{lemma}

\begin{proof}
If $L$ is a line in $S_4$ that passes through $P$, then $L$ is an irreducible component of the curve $T_P$,
because otherwise we would have
$$
1=L\cdot C=\mathrm{mult}_{P}\Big(L\cdot T_P\Big)\geqslant\mathrm{mult}_{P}\big(T_P\big)\geqslant 2,
$$
which is absurd.
\end{proof}

Put $m=\mult_P(D)$. Then Lemma~\ref{lemma:Skoda} and \eqref{equation:lambda-3-4} imply
\begin{equation}
\label{equation:m-big}
m>\frac{1}{\lambda}>\frac{4}{3}.
\end{equation}

\begin{lemma}
\label{lemma:lines}
Let $L$ be a line in $S_4$ that passes through $P$. Then $L$ is contained in $\mathrm{Supp}(D)$.
\end{lemma}

\begin{proof}
If $L$ is not contained in the support of $D$, then \eqref{equation:m-big} gives
$$
1=L\cdot H=L\cdot D\geqslant\mathrm{mult}_P(L)\mathrm{mult}_P(D)=m>\frac{1}{\lambda}>1,
$$
which is absurd.
\end{proof}

Let $f\colon\widetilde{S}_4\to S_4$ be a blow up of the surface $S$ at the point $P$.
Denote by $E$ the $f$-exceptional curve, and denote by $\widetilde{D}$ the proper transform of $D$ on the surface $\widetilde{S}_4$.
Then the log pair
\begin{equation}
\label{equation:log-pull-back}
\Big(\widetilde{S}_4,\lambda\widetilde{D}+\big(\lambda m-1\big)E\Big)
\end{equation}
is not log canonical at some point $Q\in E$ by Remark~\ref{remark:log-pull-back}.
Moreover, Lemma~\ref{lemma:log-pull-back} implies

\begin{corollary}
\label{corollary:log-pull-back}
Suppose that $m\leqslant\frac{2}{\lambda}$. Then the log pair \eqref{equation:log-pull-back} is log canonical at every point of the curve $E$ that is different from $Q$.
\end{corollary}

Put $\widetilde{m}=\mult_Q(\widetilde{D})$.
Applying Lemma~\ref{lemma:Skoda} to the log pair \eqref{equation:log-pull-back} at the point $Q$,
we obtain
\begin{equation}
\label{equation:m-m-big}
m+\widetilde{m}>\frac{2}{\lambda}>\frac{8}{3},
\end{equation}
because $\lambda<\frac{3}{4}$ by \eqref{equation:lambda-3-4}.

Let $g\colon\overline{S}_4\rightarrow \widetilde{S}_4$ be the blow up of the surface $\widetilde{S}_4$ at the point $Q$,
and let $F$ be the exceptional curve of $g$.
Denote by $\overline{E}$ and $\overline{D}$ the proper transforms of $E$ and $\widetilde{D}$, respectively.
By Remark~\ref{remark:log-pull-back}, the log pair
\begin{equation}
\label{equation:log-pull-back-2}
\Big(\overline{S}_4,\lambda\overline{D}+\big(\lambda m-1\big)\overline{E}+\big(\lambda m+\lambda\widetilde{m}-2\big)F\Big)
\end{equation}
is not log canonical at some point $O\in F$, because
$$
K_{\overline{S}_4}+\lambda\overline{D}+\big(\lambda m-1\big)\overline{E}+\big(\lambda m+\lambda\widetilde{m}-2\big)F\sim_{\mathbb{Q}} g^*\Big(K_{\widetilde{S}_4}+\lambda\widetilde{D}+\big(\lambda m-1\big)E\Big),
$$
and \eqref{equation:log-pull-back} is not log canonical at the point $Q$.
Applying Lemma~\ref{lemma:log-pull-back}, we obtain

\begin{corollary}
\label{corollary:log-pull-back-2}
Suppose that $m+\widetilde{m}\leqslant\frac{3}{\lambda}$.
Then the log pair \eqref{equation:log-pull-back-2} is log canonical at every point of $F$ that is different from $O$.
\end{corollary}

Put $\overline{m}=\mult_O(\overline{D})$.
Applying Lemma~\ref{lemma:Skoda} to the log pair \eqref{equation:log-pull-back-2} at the point $O$, we get
\begin{equation}
\label{equation:m-m-m-big}
m+\widetilde{m}+\overline{m}>\frac{3}{\lambda}>4,
\end{equation}
because $\lambda<\frac{3}{4}$ by \eqref{equation:lambda-3-4}.

Denote by $\widetilde{T}_P$ the proper transform of the singular quartic curve $T_P$ on the surface $\widetilde{S}_4$.
We have the following diagram:
\begin{center}
$\xymatrixcolsep{0.9pc}\xymatrixrowsep{1.6pc}
\xymatrix{
F\ar@{}[r]|-*[@]{\subset}\ar@{|->}[d]&\overline{S}\ar^g[d]&\\
Q\ar@{}[r]|-*[@]{\in}&\widetilde{S}\ar^f[d]&E\ar@{|->}[d]\ar@{}[l]|-*[@]{\subset}\\
&S&P\ar@{}[l]|-*[@]{\in}
}$
\end{center}
For the point $Q$, we have two mutually excluding possibilities: $Q\in\widetilde{T}_P$ and $Q\not\in\widetilde{T}_P$.
If $Q\in\widetilde{T}_P$, we can use geometry of the curve $T_P$ to derive a contradiction.
If $Q\not\in\widetilde{T}_P$, then we often can obtain a contradiction using the following two lemmas.

\begin{lemma}
\label{lemma:point-not-on-exceptional}
Suppose that  $m\leqslant\frac{2}{\lambda}$, $m+\widetilde{m}\leqslant\frac{3}{\lambda}$ and $Q\not\in\widetilde{T}_P$.
Then $O=\overline{E}\cap F$.
\end{lemma}

\begin{proof}
Suppose $O\ne\overline{E}\cap F$.
Then the linear system $|(f\circ g)^*(H)-2F-\overline{E}|$ is a free pencil.
Thus, it contains a unique curve that passes through the point $O$.
Denote this curve by $\overline{M}$, and denote its proper transform on $S_4$ by $M$.
Then $M$ is a hyperplane section of the surface $S_4$ and $P\in M$.
In particular, $M$ is reduced by Lemma~\ref{lemma:Pukhlikov}.
Since $Q\not\in\widetilde{T}_P$, we have $M\ne T_P$, so that $M$ is smooth at $P$.
Thus, $\overline{M}$ is the proper transform of the curve $M$ on the surface $\overline{S}_4$.

Since $M$ is smooth at $P$, the log pair$(S_4,\lambda M)$ is log canonical at $P$.
Thus, it follows from Remark~\ref{remark:convexity} that there exists an effective $\Q$-divisor $D^\prime$ on the surface $S_4$
such that $D^\prime\sim_{\mathbb{Q}} H$,
the log pair $(S_4,\lambda D^\prime)$ is not log canonical at $P$,
the support of the divisor $D^\prime$ is contained in the support of the divisor $D$
and does not contain at least one irreducible component of the curve $M$.
Replacing $D$ by $D^\prime$, we may assume that $D$ enjoys all these properties.

Denote by $M_\star$ the irreducible component of the curve $M$ that is not contained in the support of $D$.
Similarly, denote by $\overline{M}^\prime$ the irreducible component of the curve $\overline{M}$ that contain $O$,
and denote its image on $S_4$ by $M^\prime$. If $M_\star=M^\prime$, then
$$
\overline{m}\leqslant\overline{M}^\prime\cdot\overline{D}=\mathrm{deg}\big(M^\prime\big)-m-\widetilde{m}\leqslant 4-m-\widetilde{m},
$$
which contradicts  \eqref{equation:m-m-m-big}. Thus, we see that $M_\star\ne M^\prime$.
In particular, the curve $M$ is not irreducible.

Since $M$ is smooth at $P$ and $P\in M^\prime$, then $P\not\in M_\star$.
By Lemma~\ref{lemma:curve-C}, the curve $M^\prime$ is not a line, because $Q\not\in\widetilde{T}_P$ by assumption.
Hence, either $M^\prime$ is a conic or $M^\prime$ is a cubic curve.
Therefore, we may have  the following cases:

\begin{minipage}{0.24\textwidth}
\begin{center}
\begin{tikzpicture}[scale=0.7]
\draw (0.57,1.03) to [out=0, in=90] (2,0.3) to [out=-90, in=0](0.57,-0.5) to [out=180, in=-90] (-0.86,0.3) to [out=90, in=180] (0.57,1.03);
\draw (1.57,1.03) to [out=0, in=90] (3,0.3) to [out=-90, in=0](1.57,-0.5) to [out=180, in=-90] (0.14,0.3) to [out=90, in=180] (1.57,1.03);
\node at (0.45,1.4) {$P$};
\draw [fill] (0.42,1.03) circle [radius=0.06];
\node at (-0.8,1.1) {$M^\prime$};
\node at (3,1.1) {$M_\star$};
\node  [align=flush center,text width=4cm] at (1.2,-1) {\footnotesize $M^\prime$ and $M_\star$ are conics};
\end{tikzpicture}
\end{center}
\end{minipage}
\hfill
\begin{minipage}{0.33\textwidth}
\begin{center}
\begin{tikzpicture}[scale=0.7]
\draw (-1.17,0.4) --(2.5,0);
\draw (-1.17,0) --(2.5,0.6);
\draw (0.57,1.03) to [out=0, in=90] (2,0.3) to [out=-90, in=0](0.57,-0.5) to [out=180, in=-90] (-0.86,0.3) to [out=90, in=180] (0.57,1.03);
\node at (0.45,1.4) {$P$};
\draw [fill] (0.42,1.03) circle [radius=0.06];
\node at (-0.6,1.2) {$M'$};
\node at (2.6,0.9) {$M_\star$};
\node  [align=flush center,text width=5.5cm] at (0.7,-1) {\footnotesize $M^\prime$ is a conic, and $M_\star$ is a line};
\end{tikzpicture}
\end{center}
\end{minipage}
\hfill
\begin{minipage}{0.33\textwidth}
\begin{center}
\begin{tikzpicture}[scale=0.7]
\draw  (0.5,1.5) to [out=-100,in=120] (0.6,0.6) to [out=-60,in=-90]   (3.2,0.6) to [out=90,in=90] (2,0.6)  to [out=-90,in=100] (3,-0.7);
\node at (1.3,0.4) {$P$};
\draw (1.2,-0.5) --(3.8,1.3);
\draw [fill] (1.3,0.06) circle [radius=0.06];
\node at (1,1.2) {$M'$};
\node at (4,0.9) {$M_\star$};
\node  [align=flush center,text width=5cm] at (2.2,-1.1) {\footnotesize $M^\prime$ is a cubic, and $M_\star$ is a line};
\end{tikzpicture}
\end{center}
\end{minipage}

Put $D=aM^\prime+\Delta$, where $a$ is a non-negative rational number,
and $\Delta$ is an effective $\Q$-divisor whose support does not contain $M^\prime$.
Then $a\leqslant 1$ by Lemma~\ref{lemma:Pukhlikov}.
In fact, we can say more. Indeed, we have
$$
\mathrm{deg}\big(M_\star\big)=H\cdot M_\star=D\cdot M_\star=aM^\prime\cdot M_\star+\Delta\cdot M_\star\geqslant aM^\prime\cdot M_\star.
$$
Since $M^\prime\cdot M_\star=\mathrm{deg}(M^\prime)\mathrm{deg}(M_\star)$ on the surface $S_4$, we have
\begin{equation}
\label{equation:a-inequality}
a\leqslant\frac{\mathrm{deg}\big(M_\star\big)}{\mathrm{deg}\big(M^\prime\big)\mathrm{deg}\big(M_\star\big)}.
\end{equation}

Denote by $\widetilde{\Delta}$ the proper transform of the divisor $\Delta$ on the surface $\widetilde{S}_4$.
Put $n=\mult_P(\Delta)$ and $\widetilde{n}=\mult_Q(\widetilde{\Delta})$.
Since $O\ne\overline{E}\cap F$ and \eqref{equation:log-pull-back-2} is not log canonical at the point $O$, the log pair
$$
\Big(\overline{S}_4,\lambda a\overline{M}^\prime+\lambda\overline{\Delta}+\big(\lambda n+\lambda\widetilde{n}+2\lambda a-2\big)F\Big)
$$
is also not log canonical at the point the point $O$. Applying Theorem~\ref{theorem:adjunction} to this log pair, we obtain
$$
\overline{M}^\prime\cdot\overline{\Delta}+\big(\lambda n+\lambda\widetilde{n}+2\lambda a-2\big)=\overline{M}^\prime\cdot\Big(\lambda\overline{\Delta}+\big(\lambda n+\lambda\widetilde{n}+2\lambda a-2\big)F\Big)>1.
$$
This gives $\overline{M}^\prime\cdot\overline{\Delta}+n+\widetilde{n}+2a>\frac{3}{\lambda}$.
On the other hand, we have
$$
\overline{M}^\prime\cdot\overline{\Delta}=M^\prime\cdot\Delta-n-\widetilde{n}=M^\prime\cdot\big(H-aM^\prime)-n-\widetilde{n}=\mathrm{deg}\big(M^\prime\big)-a(M^\prime)^2-n-\widetilde{n}.
$$
Therefore, we obtain
$$
\mathrm{deg}\big(M^\prime\big)-a(M^\prime)^2>\frac{3}{\lambda}-2a>4-2a,
$$
because $\lambda>\frac{3}{4}$ by \eqref{equation:lambda-3-4}. Thus, we have
\begin{equation}
\label{equation:ugly}
a\Big(2-(M^\prime)^2\Big)>4-\mathrm{deg}\big(M^\prime\big).
\end{equation}

If $M^\prime$ is a conic, then $(M^\prime)^2=-2$, so that that $a>\frac{1}{2}$ by \eqref{equation:ugly},
which is impossible, because $a\leqslant\frac{1}{2}$ by \eqref{equation:a-inequality}.
Thus, $M^\prime$ is a plane cubic curve. Then $(M^\prime)^2=0$.
Now \eqref{equation:ugly} gives $a>\frac{1}{2}$, which is impossible, since $a\leqslant\frac{1}{3}$ by \eqref{equation:a-inequality}.
\end{proof}

\begin{lemma}
\label{lemma:O-E-F}
If $m\leqslant 2$, then $m\leqslant\frac{2}{\lambda}$, $m+\widetilde{m}\leqslant\frac{3}{\lambda}$ and $O\ne\overline{E}\cap F$.
\end{lemma}

\begin{proof}
Suppose $m\leqslant 2$.
Then $m\leqslant\frac{2}{\lambda}$, because $\lambda<\frac{3}{4}$ by \eqref{equation:lambda-3-4}.
Similarly, we see that $m+\widetilde{m}\leqslant\frac{3}{\lambda}$, because $\widetilde{m}\leqslant m$.
If $O=\overline{E}\cap F$, then
$$
\Big(\lambda\overline{D}+\big(\lambda m+\lambda\widetilde{m}-2\big)F\Big)\cdot\overline{E}>1
$$
by Theorem~\ref{theorem:adjunction}. On the other hand, we have
$$
\overline{D}\cdot\overline{E}=m-\widetilde{m}
$$
and $F\cdot\overline{E}=1$.
Hence, if $O\ne\overline{E}\cap F$, then
$2\lambda\geqslant\lambda m>\frac{3}{2}$,
which contradicts \eqref{equation:lambda-3-4}.
\end{proof}

Recall that $T_P$ is a reduced plane quartic curve that is singular at the point $P$.
This implies that there are twelve possibilities for the curve $T_P$ as follows.
\begin{enumerate}[(A)]
\item\label{four-lines}  $\mult_P(T_P)=4$, hence $T_P$ consists of four lines that intersect at $P$.
\item $\mult_P(T_P)=3$ and $T_P$
\begin{enumerate}[({B}1)]
\item\label{three-lines}  consists of four lines and three of them intersect at $P$, or
\item\label{quartic-m3} it is an irreducible quartic with a singular point $P$ of multiplicity $3$, or
\item\label{conic+2lines} it consists of a conic and  two lines, all intersecting at $P$, or
\item\label{cubic+line} it consists of a cubic curve with a singular point $P$ of multiplicity $2$ and a line passing through~$P$.
\end{enumerate}
\item $\mult_P(T_P)=2$ and $T_P$
\begin{enumerate}[({C}1)]
\item\label{2lines} consists of four lines, two of which pass through $P$, or
\item\label{2lines-conic} it consist of a conic and two lines, and the two lines intersect at $P$ and $P$ does not lie on the conic, or
\item\label{line+conic} it consist of a conic and two lines and $P$ is the intersection point of the conic with one of the lines, or
\item\label{cubic} it consists of a cubic curve and a line and $P$ is the intersection of the two at a smooth point of the cubic curve, or
\item\label{smooth-cubic+line} it consists of a cubic curve and a line and $P$ is singular point of the cubic curve with multiplicity $2$ and does not lie on the line, or
\item\label{2conics} it consists of two conics and they intersect at $P$, or
\item\label{quartic-m2} it is an irreducible quartic curve with a singular point $P$ of multiplicity $2$.
\end{enumerate}
\end{enumerate}

In the rest of this section, we  eliminate all these possibilities case by case using Lemmas~\ref{lemma:point-not-on-exceptional} and \ref{lemma:O-E-F}.
To succeed in doing this, we also need

\begin{lemma}
\label{lemma:assumption-support}
We may assume that the support of the divisor $D$ does not contain at least one irreducible component of the plane quartic curve $T_P$.
\end{lemma}

\begin{proof}
Note that $(S_4,\lambda T_P)$ is log canonical at $P$, because $\lambda<\alpha_1(S_4,H)$.
Thus, it follows from Remark~\ref{remark:convexity} that there exists an effective $\Q$-divisor $D^\prime$ on the surface $S_4$
such that $D^\prime\sim_{\mathbb{Q}} H$, the log pair $(S_4,\lambda D^\prime)$ is not log canonical at $P$,
and the support of $D^\prime$ does not contain at least one irreducible component of the curve $T_P$.
Replacing $D$ by $D^\prime$, we obtain the required assertion.
\end{proof}

We denote by $C_\star$ the irreducible component of the curve $T_P$ that is not  contained in the support of the divisor $D$.
By Lemma~\ref{lemma:lines}, if $P\in C_\star$, then $C_\star$ is not a line. This gives

\begin{corollary}
\label{corollary:A}
The case~\eqref{four-lines} is impossible.
\end{corollary}

Now we are going to deal with the cases \eqref{three-lines}, \eqref{quartic-m3}, \eqref{conic+2lines}, and \eqref{cubic+line}.
In these four cases, $\lambda<\frac{2}{3}$.
Indeed, one has $\mathrm{lct}_P(S_4,T_P)\leqslant\frac{2}{\mathrm{mult}_P(T_P)}$ by \eqref{equation:lct}.
Thus, we have
\begin{equation}
\label{equation:lambda-2-3}
\lambda<\frac{2}{\mathrm{mult}_P(T_P)},
\end{equation}
because $\lambda<\alpha_1(S_4,H)\leqslant\mathrm{lct}_P(S_4,T_P)$.

\begin{lemma}
\label{lemma:B1}
The case~\eqref{three-lines} is impossible.
\end{lemma}

\begin{proof}
Suppose that we are in the case \eqref{three-lines}.
Then $\mult_P(T_P)=3$ and $T_P$ consists of four lines $L_1$, $L_2$, $L_3$, and $L_4$ such
that the first three intersect at $P$, and $L_4$ does not pass through $P$.
Thus, we have the following picture:
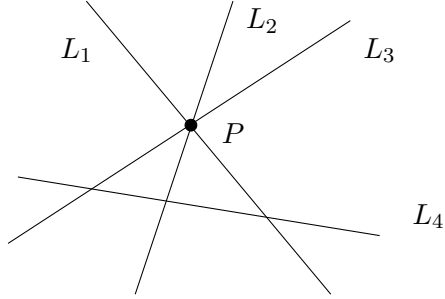
\begin{center}
\begin{tikzpicture}[scale=1.3]
\draw (0,0) --(1,3);
\draw (-0.5,3) --(2,0);
\draw (-1.2,1.2) --(2.5,0.6);
\draw (-1.3,0.52) --(2.2,2.8);
\node at (1,1.65) {$P$};
\node at (-0.6,2.5) {$L_1$};
\node at (1.3,2.8) {$L_2$};
\node at (2.5,2.5) {$L_3$};
\node at (3,0.8) {$L_4$};
\draw [fill] (0.57,1.73) circle [radius=0.06];
\end{tikzpicture}
\end{center}

By Lemma~\ref{lemma:lines}, the lines $L_1$, $L_2$, and $L_3$ are contained in the support of $D$, and $C_\star=L_4$.
Hence, we put $D=a_1L_1+a_2L_2+a_3L_3+\Omega$,
where $a_1$, $a_2$, and $a_3$ are positive rational numbers,
and $\Omega$ is an effective $\mathbb{Q}$-divisor whose support does not contain the lines $L_1$, $L_2$, $L_3$, and $L_4$.
Put $n=\mult_P(\Omega)$. Then $m=n+a_1+a_2+a_3$.

Denote by $\widetilde{\Omega}$ the proper transform of the divisor $\Omega$ on the surface $\widetilde{S}_4$.
Also denote the proper transforms of the lines  $L_1$, $L_2$, and $L_3$ on the surface $\widetilde{S}_4$ by $\widetilde{L}_1$, $\widetilde{L}_2$, and $\widetilde{L}_3$, respectively.
Then we can rewrite the log pair \eqref{equation:log-pull-back-2} as
$$
\Big(\widetilde{S}_4, \lambda a_1\widetilde{L}_1+\lambda a_2\widetilde{L}_2+\lambda a_3\widetilde{L}_3+\lambda\widetilde{\Omega}+\big(\lambda(n+a_1+a_2+a_3)-1\big)E\Big).
$$

On the surface $S_4$, one has $L_1^2=-2$. Thus, we have
$$
1=D\cdot L_1=\Big(a_1L_1+a_2L_2+a_3L_3+\Omega\Big)\cdot L_4=-2a_1+a_2+a_3+\Omega\cdot L_1\geqslant -2a_1+a_2+a_3+n.
$$
Similarly, we see that $a_1-2a_2+a_3+n\leqslant 1$ and $a_1+a_2-2a_3+n\leqslant 1$.
Adding these three inequalities together, we get $n\leqslant 1$.
On the other hand, we have
$$
1=D\cdot L_4=\Big(a_1L_1+a_2L_2+a_3L_3+\Omega\Big)\cdot L_4=a_1+a_2+a_3+\Omega\cdot L_4\geqslant a_1+a_2+a_3,
$$
which gives $a_1+a_2+a_3\leqslant 1$.
In particular, we have $m=n+a_1+a_2+a_3\leqslant 2$.
Then Lemmas~\ref{lemma:point-not-on-exceptional} and \ref{lemma:O-E-F} imply that $Q$ is contained in one of the curves $\widetilde{L}_1$, $\widetilde{L}_2$, and $\widetilde{L}_3$.
Without loss of generality, we may assume that $Q\in\widetilde{L}_1$.

As $\widetilde{L}_2$ and $\widetilde{L}_3$ do not pass through $Q$, the log pair
$(\widetilde{S}_4, \lambda a_1\widetilde{L}_1+\lambda\widetilde{\Omega}+(\lambda (n+a_1+a_2+a_3)-1)E)$
is not log canonical at the point $Q$. Moreover, we have $\mult_Q(\widetilde{\Omega})\leqslant n\leqslant 1$.
Thus, we can apply Theorem~\ref{theorem:Trento} to the log pair \eqref{equation:log-pull-back-2} and the curves $\widetilde{L}_1$ and $E$.
This gives either
\begin{multline*}
\lambda\Big(1+2a_1-a_2-a_3-n\Big)=\lambda\Big(\big(H-a_1L_1-a_2L_2-a_3L_3\big)\cdot L_1-n\Big)=\\
=\lambda\Big(\Omega\cdot L_1-n\Big)=\lambda\widetilde{\Omega}\cdot\widetilde{L}_1>2\Big(1-\big(\lambda(n+a_1+a_2+a_3)-1\big)\Big)
\end{multline*}
or $\lambda n=\lambda\widetilde{\Omega}\cdot E>2(1-\lambda a_1)$ (or both).
If the former inequality holds, then
$$
4a_1+a_2+a_3+n>\frac{4}{\lambda}-1>5,
$$
because $\lambda<\frac{2}{3}$ by \eqref{equation:lambda-2-3}.
One the other hand, we know that $a_1\leqslant 1$ by Lemma~\ref{lemma:Pukhlikov},
and we proved earlier that $a_1+a_2+a_3\leqslant 1$ and $n\leqslant 1$.
This implies that $4a_1+a_2+a_3+n\leqslant 5$.
Thus, we see that the latter inequality holds.
It gives $1+2a_1>\frac{2}{\lambda}>3$, since $\lambda<\frac{2}{3}$ by \eqref{equation:lambda-2-3}.
Thus, we conclude that $a_1>1$, which is impossible by Lemma~\ref{lemma:Pukhlikov}.
\end{proof}

\begin{lemma}
\label{lemma:B2}
The case \eqref{quartic-m3} is impossible.
\end{lemma}

\begin{proof}
Suppose that we are in the case \eqref{quartic-m3}.
Then $\mult_P(T_P)=3$ and $T_P$ is an irreducible quartic curve with a singular point $P$ of multiplicity~$3$.
Thus, we have the following picture:
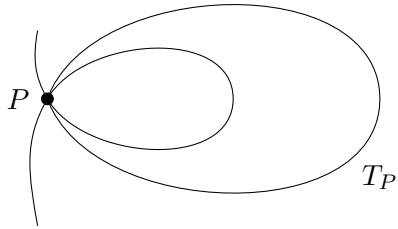
\vspace{-0.1cm}
\begin{center}
\begin{tikzpicture}[scale=1.3]
\draw  (0.5,2) to [out=-100,in=120] (0.6,1.3) to [out=-60,in=-90]  (2.5,1.3)  to [out=90,in=60] (0.6,1.3)
to [out=-70,in=-90] (4,1.3) to [out=90,in=70] (0.6,1.3)  to [out=-120,in=100] (0.5,0);
\node at (0.3,1.3) {$P$};
\node at (4,0.5) {$T_P$};
\draw [fill] (0.6,1.3) circle [radius=0.06];
\end{tikzpicture}
\end{center}

We have $C_\star=C$. Thus, it follows from \eqref{equation:m-big} that
$$
4=H\cdot C=D\cdot C\geqslant\mathrm{mult}_P(C)\mult_P(D)\geqslant 3\mult_P(D)>\frac{3}{\lambda},
$$
which contradicts \eqref{equation:lambda-3-4}.
\end{proof}

\begin{lemma}
\label{lemma:B3}
The case \eqref{conic+2lines} is impossible.
\end{lemma}

\begin{proof}
Suppose that we are in the case \eqref{conic+2lines}.
Then $\mult_P(T_P)=3$ and $T_P$ consists of a conic $C_1$ and  two lines $L_1$ and $L_2$, all intersecting at the point $P$.
Thus, we have the following picture:
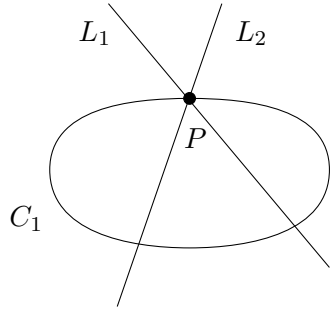
\begin{center}
\begin{tikzpicture}[scale=1.3]
\draw (-0.17,-0.4) --(0.9,2.7);
\draw (-0.26,2.7) --(2,0);
\draw (0.57,1.73) to [out=0, in=90] (2,1) to [out=-90, in=0](0.57,0.2) to [out=180, in=-90] (-0.86,1) to [out=90, in=180] (0.57,1.73);
\node at (0.63,1.32) {$P$};
\node at (1.2,2.4) {$L_2$};
\node at (-0.4,2.4) {$L_1$};
\node at (-1.1,0.5) {$C_1$};
\draw [fill] (0.57,1.73) circle [radius=0.06];
\end{tikzpicture}
\end{center}

By Lemma~\ref{lemma:lines}, both lines $L_1$ and $L_2$ are contained in the support of the divisor $D$.
Hence we can write $D=a_1L_1+a_2L_2+\Omega$, where $a_1$ and $a_2$ are positive rational numbers,
and $\Omega$ is an effective $\mathbb{Q}$-divisor whose support does not contain the lines $L_1$ and $L_2$.
Recall that the support of $\Omega$ does not contain the curve $C_\star$ by assumption.
In our case, the curve $C_\star$ is the conic $C_1$.

Put $n=\mathrm{mult}_{P}(\Omega)$. Let us show that $n\leqslant\frac{6}{5}$. We have
$$
n\leqslant\Omega\cdot L_1=\big(H-a_1L_1-a_2L_2\big)\cdot L_1=1+2a_1-a_2.
$$
Similarly, we see that $n\leqslant 1-a_1+2a_2$. Finally, we have
$$
n\leqslant\Omega\cdot C_\star=\big(H-a_1L_1-a_2L_2\big)\cdot C_\star=2-2a_1-2a_2,
$$
which implies that $a_1+a_2\leqslant 1-\frac{n}{2}$. Adding these three inequalities together, we get $n\leqslant\frac{6}{5}$.

By \eqref{equation:lambda-2-3}, we have $\lambda<\frac{2}{3}$.
Since $n\frac{6}{5}$, we see that $\lambda n\leqslant 1$.
Thus, we can apply Theorem~\ref{theorem:Trento} to the log pair $(S_4,a_1L_1+a_2L_2+\Omega)$.
This gives
$\lambda\Omega\cdot L_1>2(1-\lambda a_2)$ or $\lambda\Omega\cdot L_2>2(1-\lambda a_1)$.
Without loss of generality, we may assume that the former inequality holds. Then
$$
\lambda\big(1+2a_1-a_2\big),=\lambda\big(H-a_1L_1-a_2L_2\big)\cdot L_1=\lambda\Omega\cdot L_1>2\big(1-\lambda a_2\big),
$$
which implies that $2a_1+a_2>\frac{2}{\lambda}-1$. Since $\lambda<\frac{2}{3}$, we have $2a_1+a_2>2$,
which is impossible since we already proved that $a_1+a_2\leqslant 1-\frac{n}{2}\leqslant 1$.
\end{proof}

\begin{lemma}
\label{lemma:B4}
The case \eqref{cubic+line} is impossible.
\end{lemma}

\begin{proof}
Suppose that we are in the case \eqref{cubic+line}.
Then $\mult_P(T_P)=3$ and $T_P$ consists of a cubic curve $C_1$ with a singular point $P$ of multiplicity $2$ and a line $L$ passing through $P$.
Thus, we have the following picture:
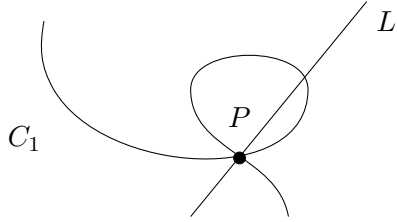
\vspace{-0.15cm}
\begin{center}
\begin{tikzpicture}[scale=1.3]
\draw  (0.5,2) to [out=-100,in=120] (0.6,1.3) to [out=-60,in=-90]   (3.2,1.3) to [out=90,in=90] (2,1.3)  to [out=-90,in=100] (3,0);
\node at (2.5,1.02) {$P$};
\node at (0.3,0.8) {$C_1$};
\node at (4,2) {$L$};
\draw (2,0) --(3.8,2.2);
\draw [fill] (2.5,0.6) circle [radius=0.06];
\end{tikzpicture}
\end{center}

By Lemma~\ref{lemma:lines}, the line $L$ is contained in the support of the divisor $D$.
Hence, $C_\star=C_1$, and we can write $D=aL+\Omega$, where $a$ is a positive rational number,
and $\Omega$ is an effective $\mathbb{Q}$-divisor whose support does not contain the line $L$.
Put $n=\mathrm{mult}_{P}(\Omega)$. Then
$$
3=H\cdot C_1=D\cdot C_1=\big(aL+\Omega)\cdot C_1=3a+\Omega\cdot C_1\geqslant 3a+2n\geqslant 2a+2n,
$$
which implies that $a+n\leqslant\frac{3}{2}$.
On the other hand, $\lambda<\frac{2}{3}$ by \eqref{equation:lambda-2-3},  so that $n+a>\frac{3}{2}$ by Lemma~\ref{lemma:Skoda}.
The contradiction is clear.
\end{proof}

\begin{lemma}
\label{lemma:C1-C2}
The cases \eqref{2lines} and \eqref{2lines-conic} are impossible.
\end{lemma}

\begin{proof}
Suppose that we are either in the case \eqref{2lines} or in the case \eqref{2lines-conic}.
Then $T_P$ consists of two lines $L_1$ and $L_2$, and a possibly reducible conic $C_1$,
where $P$ is the intersection point of the lines $L_1$ and $L_2$, and $P$ is not contained in the conic $C_1$.
If we are in the case \eqref{2lines}, then the conic $C_1$ splits as a union of two different lines $L_3$ and $L_4$,
which implies that we have the following picture:
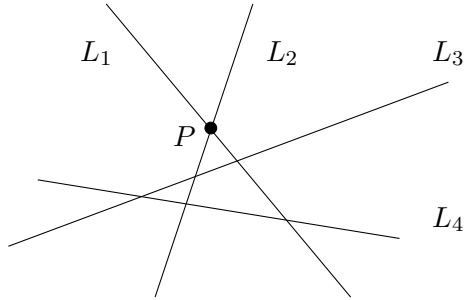
\begin{center}
\begin{tikzpicture}[scale=1.3]
\draw (0,0) --(1,3);
\draw (-0.5,3) --(2,0);
\draw (-1.2,1.2) --(2.5,0.6);
\draw (-1.5,0.52) --(3,2.2);
\node at (0.3,1.65) {$P$};
\node at (-0.6,2.5) {$L_1$};
\node at (1.3,2.5) {$L_2$};
\node at (3,2.5) {$L_3$};
\node at (3,0.8) {$L_4$};
\draw [fill] (0.57,1.73) circle [radius=0.06];
\end{tikzpicture}
\end{center}

If we are in the case \eqref{2lines-conic}, then the conic $C_1$ is irreducible, so that we have the following picture:
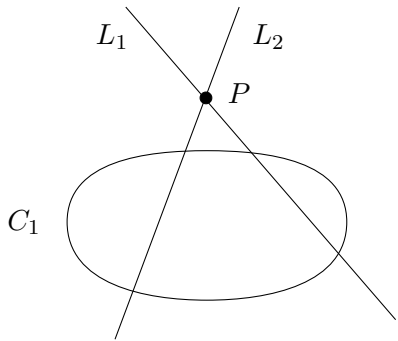
\begin{center}
\begin{tikzpicture}[scale=1.3]
\draw (-0.37,-0.7) --(0.9,2.7);
\draw (-0.26,2.7) --(2.5,-0.5);
\draw (0.57,1.23) to [out=0, in=90] (2,0.5) to [out=-90, in=0](0.57,-0.3) to [out=180, in=-90] (-0.86,0.5) to [out=90, in=180] (0.57,1.23);
\node at (0.9,1.82) {$P$};
\node at (1.2,2.4) {$L_2$};
\node at (-0.4,2.4) {$L_1$};
\node at (-1.3,0.5) {$C_1$};
\draw [fill] (0.56,1.77) circle [radius=0.06];
\end{tikzpicture}
\end{center}

By Lemma~\ref{lemma:lines}, both lines $L_1$ and $L_2$ are contained in the support of the divisor $D$.
In particular, $C_\star\ne L_1$ and $C_\star\ne L_2$.
Write $D=\Omega+a_1L_1+a_2L_2$, where $a_1$ and $a_2$ are positive rational numbers,
and $\Omega$ is an effective $\mathbb{Q}$-divisor whose support does not contain the lines $L_1$ and $L_2$.
Put $n=\mult_P(\Omega)$. Then
$$
n\leqslant\Omega\cdot L_1=\Big(H-a_1L_1-a_2L_2\Big)\cdot L_1=1+2a_1-a_2.
$$
Similarly, we see that $n\leqslant 1-a_1+2a_2$.
Finally, we have
$$
0\leqslant\Omega\cdot C_\star=\Big(H-a_1L_1-a_2L_2\Big)\cdot C_\star=\mathrm{deg}\big(C_\star\big)\big(1-a_1-a_2\big),
$$
which implies that $a_1+a_2\leqslant 1$. Adding these three inequalities together, we get $n\leqslant\frac{3}{2}$.

Recall that $m=n+a_1+a_1$. We see that $m\leqslant\frac{5}{2}$, because $a_1+a_2\leqslant 1$ and $n\leqslant\frac{3}{2}$.
In particular, $\lambda m<\frac{15}{8}$, because $\lambda<\frac{3}{4}$ by \eqref{equation:lambda-3-4}.

Denote by $\widetilde{\Omega}$ the proper transform of the divisor $\Omega$ on the surface $\widetilde{S}_4$.
Similarly, denote by $\widetilde{L}_1$ and $\widetilde{L}_2$ the proper transform of the lines $L_1$ and $L_2$ on the surface $\widetilde{S}_4$, respectively.
Then we can rewrite the log pair \eqref{equation:log-pull-back} as
$$
\Big(\widetilde{S}_4,\lambda a_1\widetilde{L}_1+\lambda a_2\widetilde{L}_2+\lambda\widetilde{\Omega}+\big(\lambda(a_1+a_2+n)-1\big)E\Big).
$$
Since  $\lambda m<\frac{15}{8}$, this log pair is log canonical at every point of $E$ that is different from $Q$ by Corollary~\ref{corollary:log-pull-back}.
Put $\widetilde{n}=\mathrm{mult}_{Q}(\widetilde{\Omega})$. Then $\widetilde{n}\leqslant n$.

Suppose that $Q\in\widetilde{L}_1$. Then $Q\not\in\widetilde{L}_2$ and
$$
\widetilde{n}\leqslant\widetilde{\Omega}\cdot\widetilde{L}_1=\Omega\cdot L_1-n=1+2a_1-a_2-n.
$$
This gives
$2\widetilde{n}\leqslant \widetilde{n}+n\leqslant 1+2a_1-a_2$, because $\widetilde{n}\leqslant n$.
Since, we already know that $n\leqslant 1-a_1+2a_2$, we get
$$
3\widetilde{n}\leqslant 2\widetilde{n}+n\leqslant 2+a_1+a_2\leqslant 3,
$$
because $a_1+a_2\leqslant 1$. Thus, we see that $\widetilde{n}\leqslant 1$.
On the other hand, the log pair
$(\widetilde{S}_4,\lambda a_1\widetilde{L}_1+\lambda\widetilde{\Omega}+(\lambda(a_1+a_2+n)-1)E)$
is not log canonical at $Q$. Thus, we can apply Theorem~\ref{theorem:Trento} to this log pair.
This gives
$$
\lambda\Big(1+2a_1-a_2-n\Big)=\lambda\Big(\Omega\cdot L_1-n\Big)=\lambda\widetilde{\Omega}\cdot\widetilde{L}_1>2\Big(1-\big(\lambda(a_1+a_2+n\big)-1\Big)
$$
or $\lambda n=\lambda\widetilde{\Omega}\cdot E>2(1-\lambda a_1)$.
Since $\lambda\leqslant\frac{3}{4}$ by \eqref{equation:lambda-3-4},
the former inequality gives
$$
n+4a_1+a_2>\frac{13}{3},
$$
which is impossible, because $n\leqslant 1+2a_2-a_1$ and $a_1+a_2\leqslant 1$.
Thus, the later inequality holds. It gives  $n+2a_1>\frac{8}{3}$.
Since $n\leqslant 1+2a_2-a_1$ and $a_1+a_2\leqslant 1$, we have $a_2>\frac{2}{3}$.
Now applying Theorem~\ref{theorem:adjunction} to the log pair \eqref{equation:log-pull-back}, we obtain
\begin{multline*}
\lambda+3\lambda a_1-1=\lambda\Big(H-a_1L_1-a_2L_2\Big)\cdot L_1+\lambda a_1+\lambda a_2-1=\lambda\Omega\cdot L_1+\lambda a_1+\lambda a_2-1=\\
=\lambda\widetilde{\Omega}\cdot\widetilde{L}_1+\lambda a_1+\lambda a_2+\lambda n-1=\Big(\lambda\widetilde{\Omega}+\big(\lambda(a_1+a_2+n)-1\big)E\Big)\cdot \widetilde{L}_1>1,
\end{multline*}
which results in $a_1>\frac{5}{9}$. On the other hand, we have $a_1+a_2\leqslant 1$ and $a_2>\frac{2}{3}$, which is absurd.

We see that $Q\not\in\widetilde{L}_1$. Similarly, we see that $Q\not\in\widetilde{L}_2$.

Recall that $m=a_1+a_1+n$. We also have $\widetilde{m}=\widetilde{n}$, because $Q\not\in\widetilde{L}_1\cup\widetilde{L}_2$.
Earlier, we proved that $a_1+a_2\leqslant 1$ and $n\leqslant\frac{3}{2}$.
In particular, we have $\widetilde{n}\leqslant\frac{3}{2}$ as well, because $\widetilde{n}\leqslant n$.
Thus, we have
$$
m+\widetilde{m}=a_1+a_2+n+\widetilde{n}\leqslant a_1+a_2+2n\leqslant 4<\frac{3}{\lambda},
$$
because $\lambda<\frac{3}{4}$ by \eqref{equation:lambda-3-4}.
Thus, it follows from Corollary~\ref{corollary:log-pull-back-2} that
the log pair \eqref{equation:log-pull-back-2} is log canonical at every point of $F$ that is different from $O$.
Moreover, we have $O=F\cap\overline{E}$ by Lemma~\ref{lemma:point-not-on-exceptional},
because  $m<\frac{2}{\lambda}$, $m+\widetilde{m}<\frac{3}{\lambda}$, and $Q\not\in\widetilde{L}_1\cup\widetilde{L}_2$.

Denote by $\overline{\Omega}$ the proper transform of the divisor $\Omega$ on the surface $\overline{S}_4$.
Since $Q\not\in\widetilde{L}_1\cup\widetilde{L}_2$, the log pair
$$
\Big(\overline{S}_4, \lambda\overline{\Omega}+\big(\lambda(a_1+a_2+n)-1\big)\overline{E}+\big(\lambda(a_1+a_2+n+\widetilde{n})-2\big)F\Big)
$$
is not log canonical at the point $O$ and is log canonical at every point of $F$ that is different from $O$.
Applying Theorem~\ref{theorem:adjunction} to this log pair and the curve $\overline{E}$, we get
\begin{multline*}
\lambda\big(a_1+a_2+2n)-2=\lambda\big(n-\widetilde{n}\big)+\lambda(a_1+a_2+n+\widetilde{n})-2=\\
=\lambda\overline{\Omega}\cdot\overline{E}+\lambda(a_1+a_2+n+\widetilde{n})-2=\Big(\lambda\overline{\Omega}+\big(\lambda(a_1+a_2+n+\widetilde{n})-2\big)F\Big)\cdot\overline{E}>1
\end{multline*}
which implies that $a_1+a_2+2n>\frac{3}{\lambda}>4$, because $\lambda<\frac{3}{4}$ by \eqref{equation:lambda-3-4}.
This is a contradiction, since we already proved that $a_1+a_2\leqslant 1$ and $n\leqslant\frac{3}{2}$.
\end{proof}

\begin{lemma}
\label{lemma:C3}
The case \eqref{line+conic} is impossible.
\end{lemma}

\begin{proof}
Suppose that we are in the case \eqref{line+conic}.
Then  $\mult_P(T_P)=2$, the curve $T_P$ consist of a conic curve $C_1$ and two lines $L_1$ and $L_2$,
and the point $P$ is the intersection point of the conic with the line $L_1$.
Thus, we have the following picture:
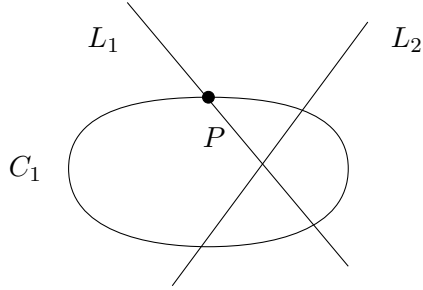
\begin{center}
\begin{tikzpicture}[scale=1.3]
\draw (0.2,-0.2) --(2.2,2.5);
\draw (-0.26,2.7) --(2,0);
\draw (0.57,1.73) to [out=0, in=90] (2,1) to [out=-90, in=0](0.57,0.2) to [out=180, in=-90] (-0.86,1) to [out=90, in=180] (0.57,1.73);
\node at (0.63,1.32) {$P$};
\node at (-0.5,2.32) {$L_1$};
\node at (2.6,2.32) {$L_2$};
\node at (-1.3,1) {$C_1$};
\draw [fill] (0.57,1.73) circle [radius=0.06];
\end{tikzpicture}
\end{center}

By Lemma~\ref{lemma:lines}, the line $L_1$ is contained in the support of the divisor $D$.
In particular, $C_\star\ne L_1$. Thus, either $C_\star=L_2$ of $C_\star=C_1$.
Write $D=\Omega+aL_1+bC_1$, where $a$ is a positive rational number, $b$ is a non-negative rational number,
and $\Omega$ is an effective $\mathbb{Q}$-divisor whose support does not contain the curves $L_1$ and $C_1$.
If $b>0$, then the support of $\Omega$ does not contain the line $L_2$, which implies that
$$
1-a-2b=\Big(H-aL_1-bC_1\Big)\cdot L_2=\Omega\cdot L_2\geqslant 0.
$$
Hence, either $b=0$ or $a+2b\leqslant 1$ (or both), so that $a+2b\leqslant 1$, because $a\leqslant 1$ by Lemma~\ref{lemma:Pukhlikov}.

Put $n=\mult_P(\Omega)$. Then
$$
n\leqslant\Omega\cdot L_1=\Big(H-aL_1-bC_1\Big)\cdot L_1=1+2a-2b.
$$
Similarly, we see that
$$
n\leqslant\Omega\cdot C_1=\Big(H-aL_1-bC_1\Big)\cdot C_1=2-2a+2b.
$$
Adding these inequalities, we get $n\leqslant\frac{3}{2}$.
This gives
$m=n+a+b\leqslant n+a+2b\leqslant\frac{5}{2}<\frac{2}{\lambda}$,
because $\lambda>\frac{3}{4}$ by \eqref{equation:lambda-3-4}.

Denote by $\widetilde{\Omega}$ the proper transform of the divisor $\Omega$ on the surface $\widetilde{\Omega}$.
Similarly, denote by $\widetilde{L}_1$ and $\widetilde{C}_1$ the proper transform of the curves $L_1$ and $C_1$ on the surface $\widetilde{\Omega}$, respectively.
Then we can rewrite the log pair \eqref{equation:log-pull-back} as
$$
\Big(\widetilde{S}_4,\lambda a\widetilde{L}_1+\lambda b\widetilde{C}_1+\lambda\widetilde{\Omega}+\big(\lambda(a+b+n)-1\big)E\Big).
$$
Since  $m<\frac{2}{\lambda}$, this log pair is log canonical at every point of $E$ that is different from $Q$ by Corollary~\ref{corollary:log-pull-back}.
Put $\widetilde{n}=\mathrm{mult}_{Q}(\widetilde{\Omega})$. Then $\widetilde{n}\leqslant n$.

Let us show that $Q\not\in\widetilde{L}_1$. Suppose that $Q\in\widetilde{L}_1$. Then
$$
\widetilde{n}\leqslant\widetilde{\Omega}\cdot\widetilde{L}_1=\Omega\cdot L_1-n=1+2a-2b-n,
$$
which implies that $2\widetilde{n}\leqslant\widetilde{n}+n\leqslant 1+2a-2b$.
But we already know that $\widetilde{n}\leqslant n\leqslant 2-2a+2b$.
Adding these two inequalities together, we get $\widetilde{n}\leqslant 1$.
If $Q\in\widetilde{C}_1$, then we also have
$$
\widetilde{n}\leqslant\widetilde{\Omega}\cdot\widetilde{C}_1=\Omega\cdot C_1-n=2-2a+2b-n,
$$
which implies that $2\widetilde{n}\leqslant\widetilde{n}+n\leqslant 2-2a+2b$.
Thus, if $Q\in\widetilde{C}_1$, then
$$
\widetilde{n}\leqslant\frac{1}{4}\Big(\big(1+2a-2b\big)+\big(2-2a+2b\big)\Big)\leqslant\frac{3}{4}.
$$
Keeping in mind that $a+2b\leqslant 1$, we conclude that $\widetilde{n}+b\leqslant\frac{5}{4}$ provided that $Q\in\widetilde{C}_1$.
In particular, the multiplicity of the $\mathbb{Q}$-divisor
$\lambda b\widetilde{C}_1+\lambda\widetilde{\Omega}$ at the point $Q$ does not exceed $1$, since $\lambda<\frac{3}{4}$ by \eqref{equation:lambda-3-4}.
Hence, we can apply Theorem~\ref{theorem:Trento} to \eqref{equation:log-pull-back} and the curves $E$ and $\widetilde{L}_1$.
This gives either
$$
\lambda+2\lambda a-\lambda b-\lambda n=\big(\lambda b\widetilde{C}_1+\lambda\widetilde{\Omega}\big)\cdot\widetilde{L}_1>2\Big(1-\big(\lambda(a+b+n\big)-1\Big)
$$
or
$$
\lambda b+\lambda n=\lambda b+\lambda\widetilde{\Omega}\cdot E=\big(\lambda\widetilde{C}_1+\lambda\widetilde{\Omega}\big)\cdot E>2\big(1-\lambda a\big)
$$
(or both). Since $\lambda<\frac{3}{4}$ by \eqref{equation:lambda-3-4},
this gives either $4a+b+n>\frac{13}{3}$ or $2a+b+n>\frac{8}{3}$ (or both).
On the other hand, we already proved that $n\leqslant 2-2a+2b$ and $a+2b\leqslant 1$. Thus, we have
$$
4a+b+n=\big(2a-2b+n\big)+2\big(a+2b\big)\leqslant 4<\frac{13}{3},
$$
which implies that $2a+b+n>\frac{8}{3}$. This gives
$$
\frac{8}{3}<2a+b+n\leqslant 2+3b,
$$
because $n\leqslant 2-2a+2b$. Hence, we obtain $b>\frac{2}{9}$.
On the other hand, applying Theorem~\ref{theorem:adjunction} to the log pair \eqref{equation:log-pull-back} and the curve $\widetilde{L}_1$, we obtain
\begin{multline*}
\lambda+3\lambda a-1=\lambda\big(\Omega\cdot L_1-n\big)+\lambda a+2\lambda b+\lambda n-1=\lambda\widetilde{\Omega}\cdot\widetilde{L}_1+\lambda a+2\lambda b+\lambda n-1=\\
=\Big(\lambda b\widetilde{C}_1+\lambda\widetilde{\Omega}+\big(\lambda(a+b+n)-1\big)E\Big)\cdot \widetilde{L}_1>1,
\end{multline*}
which results in $a>\frac{2}{\lambda}-1$.
Since $\lambda>\frac{3}{4}$, we have $a>\frac{5}{9}$.
But $a+2b\leqslant 1$, so that $b\leqslant\frac{2}{9}$.
The obtained contradiction shows that the curve $\widetilde{L}_1$ does not contain the point $Q$.

Let us show that  the curve $\widetilde{C}_1$ does not contain the point $Q$.
Indeed, suppose it does. Then
$$
\widetilde{n}\leqslant\widetilde{\Omega}\cdot\widetilde{C}_1=\Omega\cdot C_1-n=2-2a+2b-n,
$$
which implies that $2\widetilde{n}\leqslant\widetilde{n}+n\leqslant 2-2a+2b$.
But $\widetilde{n}\leqslant n\leqslant\Omega\cdot L_1=1+2a-2b$, we see that
$$
3\widetilde{n}\leqslant\big(1+2a-2b\big)+\big(2-2a+2b\big)=3,
$$
which implies $\widetilde{n}\leqslant 1$.
On the other hand, the log pair
$(\widetilde{S}_4,\lambda b\widetilde{C}_1+\lambda\widetilde{\Omega}+(\lambda(a+b+n)-1)E)$
is not log canonical at the point $Q$, because $Q\not\in\widetilde{L}_1$.
Moreover, we can apply Theorem~\ref{theorem:Trento} to this log pair, because $\widetilde{n}\leqslant 1$ and $\lambda<\frac{3}{4}$.
This gives
$$
\lambda\Big(2-2a+2b-n\Big)=\lambda\Big(\Omega\cdot C_1-n\Big)=\lambda\widetilde{\Omega}\cdot\widetilde{C}_1>2\Big(1-\big(\lambda(a+b+n\big)-1\Big)
$$
or $\lambda n=\lambda\widetilde{\Omega}\cdot E>2(1-\lambda b)$.
The former inequality gives $4b+n>\frac{4}{\lambda}-2$,
and the later inequality gives $2b+n>\frac{2}{\lambda}$.
Since $\lambda<\frac{3}{4}$, we see that either $4b+n>\frac{10}{3}$ or $2b+n>\frac{8}{3}$ (or both).
But $n\leqslant\Omega\cdot L_1=1+2a-2b$ and $a+2b\leqslant 1$, which implies that
$$
4b+n\leqslant 1+2a+2b\leqslant 3<\frac{10}{3}.
$$
Thus, we have $2b+n>\frac{8}{3}$.
One the other hand, we already know that $n+2b-2a\leqslant 1$,  $n+2b-2a\leqslant 2$, and $a+2b\leqslant 1$,
so that
$$
n+2b=\frac{2}{3}\big(n+2b-2a\big)+\frac{1}{3}\big(n+2b-2a\big)+\frac{2}{3}\big(a+2b\big)\leqslant \frac{2}{3}+\frac{2}{3}+\frac{2}{3}=2,
$$
which is a contradiction. This shows that $Q\not\in\widetilde{C}_1$.

Denote by $\overline{\Omega}$ the proper transform of the divisor $\Omega$ on the surface $\overline{S}_4$.
Recall that the log pair \eqref{equation:log-pull-back-2} is not log canonical at the point $O\in F$.
Moreover, it is log canonical at every point of $F$ that is different from $O$ by Corollary~\ref{corollary:log-pull-back-2}, because
$$
m+\widetilde{m}=a+b+n+\widetilde{n}\leqslant a+2b+2n\leqslant 4<\frac{3}{\lambda},
$$
since $a+2b\leqslant 1$, $n\leqslant\frac{3}{2}$ and $\lambda<\frac{3}{4}$.
Then $O=F\cap\overline{E}$ by Lemma~\ref{lemma:point-not-on-exceptional}.

Since $Q\not\in\widetilde{L}_1\cup\widetilde{C}_1$, we see that the log pair
$$
\Big(\overline{S}_4, \lambda\overline{\Omega}+\big(\lambda(a+b+n)-1\big)\overline{E}+\big(\lambda(a+b+n+\widetilde{n})-2\big)F\Big)
$$
is not log canonical at the point $O\in F$ and is log canonical in all other points of the curve $F$.
Applying Theorem~\ref{theorem:adjunction} to this log pair and the curve $\overline{E}$, we get
\begin{multline*}
\lambda\big(a+b+2n)-2=\lambda\big(n-\widetilde{n}\big)+\lambda(a+b+n+\widetilde{n})-2=\\
=\lambda\overline{\Omega}\cdot\overline{E}+\lambda(a+b+n+\widetilde{n})-2=\Big(\lambda\overline{\Omega}+\big(\lambda(a+b+n+\widetilde{n})-2\big)F\Big)\cdot\overline{E}>1
\end{multline*}
which implies that $a+b+2n>\frac{3}{\lambda}>4$.
On the other hand,  $n+2b-2a\leqslant 1$,  $n+2b-2a\leqslant 2$l and $a+2b\leqslant 1$. Thus, we have
$$
n+a+b=\frac{11}{12}\big(n+2b-2a\big)+\frac{13}{12}\big(n+2b-2a\big)+\frac{2}{3}\big(a+2b\big)\leqslant \frac{11}{12}+\frac{13}{6}+\frac{2}{3}=\frac{15}{4}<4,
$$
which is a contradiction.
\end{proof}

\begin{lemma}
\label{lemma:C4}
The case \eqref{cubic} is impossible.
\end{lemma}

\begin{proof}
Suppose that we are in the case \eqref{cubic}.
Then $\mult_P(T_P)=2$ and $T_P$ consists of a cubic curve $C_1$ and a line $L$, and $P$ is their intersection at a smooth point of the cubic curve.
Thus, we have the following picture:
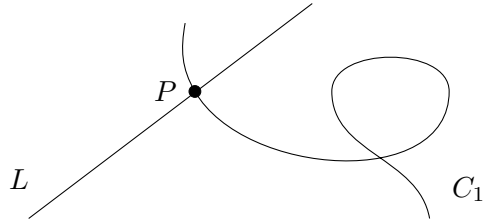
\vspace{-0.05cm}
\begin{center}
\begin{tikzpicture}[scale=1.3]
\draw  (0.5,2) to [out=-100,in=120] (0.6,1.3) to [out=-60,in=-90]   (3.2,1.3) to [out=90,in=90] (2,1.3)  to [out=-90,in=100] (3,0);
\node at (0.3,1.3) {$P$};
\node at (3.4,0.3) {$C_1$};
\node at (-1.2,0.4) {$L$};
\draw (-1.1,0) --(1.8,2.2);
\draw [fill] (0.6,1.3) circle [radius=0.06];
\end{tikzpicture}
\end{center}

By Lemma~\ref{lemma:lines}, the line $L_1$ is contained in the support of the divisor $D$, so that $C_\star=C_1$.
Write $D=\Omega+aL_1$, where $a$ is a positive rational number,
and $\Omega$ is an effective $\mathbb{Q}$-divisor whose support does not contain the line $L_1$.
Put $n=\mult_P(\Omega)$. Then
$$
n\leqslant\Omega\cdot L_1=\Big(H-aL_1\Big)\cdot L_1=1+2a,
$$
which gives $n-2a\leqslant 1$. Similarly, we obtain  $n+3a\leqslant 3$, because
$$
n\leqslant\Omega\cdot C_1=\Big(H-aL_1\Big)\cdot C_1=3-3a.
$$
We see that $n+a=\frac{2}{5}(n-2a)+\frac{3}{5}(n+3a)\leqslant\frac{11}{5}$,
which implies that $m=n+a<\frac{2}{\lambda}$, because $\lambda>\frac{3}{4}$.
Thus, it follows from Corollary~\ref{corollary:log-pull-back} that
the log pair \eqref{equation:log-pull-back} is log canonical at every point of $E$ that is different from $Q$.

Note that $a\leqslant 1$ by Lemma~\ref{lemma:Pukhlikov}. This also follows from $n+3a\leqslant 3$.
We also know that $a>0$. In fact, one can show that $a>\frac{1}{6}$.
Indeed, we have $\lambda\big(1+2a\big)=\lambda\Omega\cdot L_1>1$ by Theorem~\ref{theorem:adjunction}.
This gives $a>\frac{1}{6}$, since $\lambda>\frac{3}{4}$.

Denote by $\widetilde{\Omega}$ the proper transform of the divisor $\Omega$ on the surface $\widetilde{\Omega}$.
Similarly, denote by $\widetilde{L}_1$ the proper transform of the line $L_1$ on the surface $\widetilde{\Omega}$.
Then we can rewrite the log pair \eqref{equation:log-pull-back} as
$(\widetilde{S}_4,\lambda a\widetilde{L}_1+\lambda\widetilde{\Omega}+(\lambda(a+n)-1)E)$.
Put $\widetilde{n}=\mathrm{mult}_{Q}(\widetilde{\Omega})$. Then $\widetilde{n}\leqslant n$.

Suppose that $Q\in\widetilde{L}_1$. Then
$$
\widetilde{n}\leqslant\widetilde{\Omega}\cdot\widetilde{L}_1=\Omega\cdot L_1-n=1+2a-n,
$$
which implies that $2\widetilde{n}\leqslant\widetilde{n}+n\leqslant 1+2a$.
Since  $\widetilde{n}\leqslant n$ and $n+3a\leqslant 3$, we have $\widetilde{n}+3a\leqslant 3$.
Thus, we have $8\widetilde{n}=2(\widetilde{n}+3a)+3(2\widetilde{n}-2a)\leqslant 9$,
which gives $\widetilde{n}\leqslant\frac{9}{8}$. Then $\lambda\widetilde{n}\leqslant 1$.
Hence, we can apply Theorem~\ref{theorem:Trento} to the log pair \eqref{equation:log-pull-back} and the curves $E$ and $\widetilde{L}_1$.
This gives
$$
\lambda+2\lambda a-\lambda n=\lambda\widetilde{\Omega}\cdot\widetilde{L}_1>2\Big(1-\big(\lambda(a+n\big)-1\Big)
$$
or $\lambda n=\lambda\widetilde{\Omega}\cdot E>2(1-\lambda a)$.
Since $\lambda\leqslant\frac{3}{4}$ by \eqref{equation:lambda-3-4}, the former inequality gives
$n+4a>\frac{4}{\lambda}-1>\frac{13}{3}$,
and the later inequality gives
$n+2a>\frac{4}{\lambda}>\frac{8}{3}$.
Each of these inequalities leads to a contradiction, because $n-2a\leqslant 1$ and $n+3a\leqslant 3$.
Indeed, we have
$$
n+2a=\frac{1}{5}\big(n-2a\big)+\frac{4}{5}\big(n+3a\big)\leqslant \frac{1}{5}+\frac{12}{5}=\frac{13}{5}<\frac{8}{3}.
$$
Similarly, $n+4a\leqslant n+3a\leqslant 3\leqslant\frac{13}{3}$.
This shows that $\widetilde{L}_1$ does not contain the point $Q$.

Let us show that  $Q\not\in \widetilde{C}_1$.
Suppose $Q\in\widetilde{C}_1$. Then
$$
3-3a-n=\Omega\cdot C_1-n=\widetilde{\Omega}\cdot\widetilde{C}_1\geqslant\widetilde{n},
$$
which implies $n+a+\widetilde{n}\leqslant 3-2a$. Thus, we have
$$
3-2a\geqslant a+n+\widetilde{n}=m+\widetilde{m}>\frac{8}{3}
$$
by \eqref{equation:m-m-big}. This gives $a<\frac{1}{6}$. But we already proved that $a>\frac{1}{6}$.
This shows that $Q\not\in\widetilde{C}_1$.

Recall that $n-2a\leqslant 1$ and $n+3a\leqslant 3$.
Adding these two inequalities together, we obtain
$m+\widetilde{m}=a+n+\widetilde{n}\leqslant a+2n\leqslant 4<\frac{3}{\lambda}$, since $\lambda<\frac{3}{4}$.
Thus, Corollary~\ref{corollary:log-pull-back-2} implies that the log pair \eqref{equation:log-pull-back-2}
is log canonical at every point of the curve $F$ that is different from $O$.
By Lemma~\ref{lemma:point-not-on-exceptional}, we have $O=F\cap\overline{E}$, because
$m<\frac{2}{\lambda}$, $m+\widetilde{m}<\frac{3}{\lambda}$ and $Q\not\in\widetilde{L}_1\cup\in\widetilde{C}_1$.

Denote by $\overline{\Omega}$ the proper transform of the divisor $\Omega$ on the surface $\overline{S}_4$.
Then the log pair
$(\overline{S}_4, \lambda\overline{\Omega}+(\lambda(a+n)-1)\overline{E}+(\lambda(a+n+\widetilde{n})-2)F)$
coincides with the log pair \eqref{equation:log-pull-back-2} in a neighborhood of the point $O$, because $Q\not\in\widetilde{L}_1$.
Applying Theorem~\ref{theorem:adjunction} to this log pair and the curve $\overline{E}$, we get
$$
\lambda\big(a+2n)-2=\lambda\overline{\Omega}\cdot\overline{E}+\lambda(a+n+\widetilde{n})-2=\Big(\lambda\overline{\Omega}+\big(\lambda(a+n+\widetilde{n})-2\big)F\Big)\cdot\overline{E}>1
$$
which implies that $a+2n>\frac{3}{\lambda}$. But we already proved that $n-2a\leqslant 1$ and $n+3a\leqslant 3$.
Thus, we have
$a+2n\leqslant 4<\frac{3}{\lambda}$, because $\lambda>\frac{3}{4}$.
This is a contradiction.
\end{proof}

\begin{lemma}
\label{lemma:C5}
The case \eqref{smooth-cubic+line} is impossible.
\end{lemma}

\begin{proof}
Suppose that we are in the case \eqref{smooth-cubic+line}.
Then $\mult_P(T_P)=2$ and $T_P$ consists of a cubic curve $C_1$ and a line $L$ such that
$P$ is a singular point of the cubic curve with multiplicity $2$ and does not lie on the line $L$.
Thus, we have the following picture:
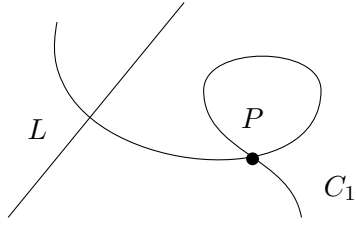
\begin{center}
\begin{tikzpicture}[scale=1.3]
\draw  (0.5,2) to [out=-100,in=120] (0.6,1.3) to [out=-60,in=-90]   (3.2,1.3) to [out=90,in=90] (2,1.3)  to [out=-90,in=100] (3,0);
\node at (2.5,1.02) {$P$};
\node at (3.4,0.3) {$C_1$};
\node at (0.3,0.9) {$L$};
\draw (0,0) --(1.8,2.2);
\draw [fill] (2.5,0.6) circle [radius=0.06];
\end{tikzpicture}
\end{center}

Write $D=\Omega+aC_1$, where $a$ is a non-negative rational number,
and $\Omega$ is an effective $\mathbb{Q}$-divisor whose support does not contain the curve $C_1$.
Put $n=\mult_P(\Omega)$. Then $m=n+2a$. If $a>0$, then $C_\star=L_1$, so that
$$
1=D\cdot L_1=\big(\Omega+aC_1\big)\cdot L_1=\Omega\cdot L_1+3a\geqslant 3a,
$$
because $C_\star$ is not contained in the support of the divisor $D$.
Hence, we see that $a\leqslant\frac{1}{3}$.
On the other hand, we have
$$
2n=\mathrm{mult}_{P}\big(C_1\big)\leqslant\Omega\cdot C_1=\big(H-aC_1\big)\cdot C_1=3.
$$
Thus, we have $n\leqslant\frac{3}{2}$.
Then $m=n+2a<\frac{2}{\lambda}$,
because $\lambda>\frac{3}{4}$ by \eqref{equation:lambda-3-4}.
Thus, it follows from Corollary~\ref{corollary:log-pull-back} that
the log pair \eqref{equation:log-pull-back} is log canonical at every point of $E$ that is different from $Q$.

Denote by $\widetilde{\Omega}$ the proper transform of the divisor $\Omega$ on the surface $\widetilde{\Omega}$.
Similarly, denote by $\widetilde{C}_1$ the proper transform of the curve $L_1$ on the surface $\widetilde{\Omega}$.
Then we can rewrite the log pair \eqref{equation:log-pull-back}  as
$(\widetilde{S}_4,\lambda a\widetilde{C}_1+\lambda\widetilde{\Omega}+(\lambda(n+2a)-1)E)$.
Put $\widetilde{n}=\mathrm{mult}_{Q}(\widetilde{\Omega})$. Then $\widetilde{n}\leqslant n$.
If $Q\not\in\widetilde{C}_1$, then $\widetilde{m}=\widetilde{n}$.
If $Q\in\widetilde{C}_1$, then $\widetilde{m}=\widetilde{n}+a$.

Denote by $\overline{\Omega}$ the proper transform of the divisor $\Omega$ on the surface $\overline{S}_4$,
and denote by $\overline{C}_1$ the proper transform of the curve $C_1$ on the surface $\overline{S}_4$.
Then we can rewrite the log pair \eqref{equation:log-pull-back-2} as
$(\overline{S}_4, \lambda a\overline{C}_1+\lambda\overline{\Omega}+(\lambda(n+2a)-1)\overline{E}+(\lambda(n+2a+\widetilde{m})-2)F)$.
This log pair is not log canonical at the point $O\in F$ by construction.
Moreover, we have
$$
m+\widetilde{m}=n+2a+\widetilde{n}+a\leqslant 2n+3a\leqslant 3+3a\leqslant 4<\frac{3}{\lambda},
$$
since $\lambda<\frac{3}{4}$.
Thus, it follows from Corollary~\ref{corollary:log-pull-back-2} that
the log pair \eqref{equation:log-pull-back-2} is log canonical at every point of the curve $F$
that is different from the point $O$.

Let us show that $O\ne F\cap\overline{E}$. Suppose that $O=F\cap\overline{E}$.
If $O\not\in\overline{C}_1$, then Theorem~\ref{theorem:adjunction} applied to the log pair \eqref{equation:log-pull-back-2}
and the curve $\overline{E}$ gives
\begin{multline*}
\lambda\big(3a+2n)-2\geqslant\lambda\big(2a+2n+\widetilde{m}-\widetilde{n})-2=\lambda\big(n-\widetilde{n}\big)+\lambda(n+2a+\widetilde{m})-2=\\
=\lambda\overline{\Omega}\cdot\overline{E}+\lambda(n+2a+\widetilde{m})-2=\Big(\lambda\overline{\Omega}+\big(\lambda(n+2a+\widetilde{m})-2\big)F\Big)\cdot\overline{E}>1
\end{multline*}
which implies that $3a+2n>\frac{3}{\lambda}$.
This is impossible, because $a\leqslant\frac{1}{3}$, $n\leqslant\frac{3}{2}$ and $\lambda\leqslant\frac{3}{4}$.
Thus, we see that $O\in\overline{C}_1$.
In particular, $Q\in\widetilde{C}_1$, $\widetilde{m}=\widetilde{n}+a$, and $C_1$ has a cuspidal singularity at the point $P$.
Now we apply Theorem~\ref{theorem:adjunction} to the log pair \eqref{equation:log-pull-back-2} and the curve $\overline{C}_1$ at the point $O$.
This gives
\begin{multline*}
\lambda\big(3+5a)-3=\lambda\big(\Omega\cdot C_1+5a\big)-3=\lambda\big(\widetilde{\Omega}\cdot\widetilde{C}_1-\widetilde{n})+\lambda(2n+5a+\widetilde{n})-3=\\
=\Big(\lambda\overline{\Omega}+\big(\lambda(n+2a)-1\big)\overline{E}+\big(\lambda(n+3a+\widetilde{n})-2\big)F\Big)\cdot\overline{C}_1>1
\end{multline*}
which implies that $5a>\frac{4}{\lambda}-3$.
Since $\lambda\leqslant\frac{3}{4}$, we have
$a>\frac{1}{5}(\frac{4}{\lambda}-3)>\frac{7}{15}$,
which is impossible, because we already proved that $a\leqslant\frac{1}{3}$.
Thus, we see that $O\ne F\cap\overline{E}$.

We already know that $m<\frac{2}{\lambda}$ and $m+\widetilde{m}<\frac{3}{\lambda}$.
Thus, if $Q\not\in\widetilde{C}_1$, then we can apply Lemma~\ref{lemma:point-not-on-exceptional} to obtain $O=F\cap\overline{E}$,
which is not  the case.
Hence, we conclude that $Q\in\widetilde{C}_1$, so that $\widetilde{m}=\widetilde{n}+a$.
If $O\not\in\overline{C}_1$, then the log pair
$(\overline{S}_4, \lambda\overline{\Omega}+(\lambda(n+2a+\widetilde{m})-2)F)$
is not  log canonical at the point $O$ as well, which implies that
$\widetilde{n}=\overline{\Omega}\cdot F>\frac{1}{\lambda}>\frac{4}{3}$ by Theorem~\ref{theorem:adjunction}.
On the other hand, we have
$$
3=\Omega\cdot C_1-2n=\widetilde{\Omega}\cdot\widetilde{C}_1\geqslant\widetilde{n},
$$
which implies that $3\widetilde{n}\leqslant 2n+\widetilde{n}\leqslant 3$, so that $\widetilde{n}\leqslant 1$.
This shows that $O\in\overline{C}_1$.

Since $O\ne F\cap\overline{E}$ and $O\in\overline{C}_1$, we conclude that $P$ is an ordinary double point of the curve $C_1$.
Hence, the curves $\widetilde{C}_1$ and $E$ intersect transversally at the point $Q$.
Thus, applying Theorem~\ref{theorem:adjunction} to the log pair \eqref{equation:log-pull-back} and the curve $E$, we get
$\lambda n=\lambda\widetilde{\Omega}\cdot E>1-\lambda a$,
which implies $a+n>\frac{1}{\lambda}>\frac{4}{3}$.
Similarly, applying Theorem~\ref{theorem:adjunction} to the log pair \eqref{equation:log-pull-back} and the curve $\widetilde{C}_1$, we get
$$
\lambda\big(3-2n)=\lambda\widetilde{\Omega}\cdot\widetilde{C}_1>1-\big(\lambda(2a+n)-1\big)=2-\lambda(2a+n),
$$
which implies that $2a>n+\frac{2}{\lambda}-3>n-\frac{1}{3}$.
Thus, we have
$2a>n-\frac{1}{3}>(\frac{4}{3}-a)-\frac{1}{3}=1-a$,
which implies that $a>\frac{1}{3}$. But we already proved that $a\leqslant\frac{1}{3}$.
This is a contradiction.
\end{proof}

\begin{lemma}
\label{lemma:C6}
The case \eqref{2conics} is impossible.
\end{lemma}

\begin{proof}
Suppose that we are in the case \eqref{2conics}.
Then $\mult_P(T_P)=2$ and $T_P$ consists of two conic curves and they intersect at $P$.
Thus, we have the following picture:
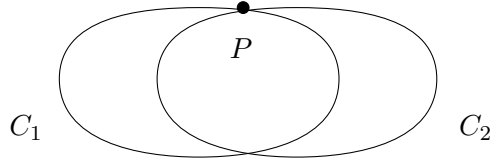
\begin{center}
\begin{tikzpicture}[scale=1.3]
\draw (0.57,1.73) to [out=0, in=90] (2,1) to [out=-90, in=0](0.57,0.2) to [out=180, in=-90] (-0.86,1) to [out=90, in=180] (0.57,1.73);
\draw (1.57,1.73) to [out=0, in=90] (3,1) to [out=-90, in=0](1.57,0.2) to [out=180, in=-90] (0.14,1) to [out=90, in=180] (1.57,1.73);
\node at (1,1.32) {$P$};
\node at (-1.2,0.5) {$C_1$};
\node at (3.4,0.5) {$C_2$};
\draw [fill] (1.02,1.73) circle [radius=0.06];
\end{tikzpicture}
\end{center}

Without loss of generality, we may assume that $C_1=C_\star$. This gives  $2=C_1\cdot D\geqslant m$.
Then $m\leqslant\frac{2}{\lambda}$ and $m+\widetilde{m}\leqslant\frac{3}{\lambda}$ by Lemma~\ref{lemma:O-E-F}.
Hence, Corollary~\ref{corollary:log-pull-back} implies that the log pair \eqref{equation:log-pull-back} is log canonical at every point of the curve $E$ that is different from $Q$.
Moreover, Corollary~\ref{corollary:log-pull-back-2} implies that the log pair \eqref{equation:log-pull-back-2} is log canonical at every point of the curve $F$ that is different from $O$.
Furthermore, Lemma~\ref{lemma:O-E-F} implies that $O\ne\overline{E}\cap F$.

Denote by $\widetilde{C}_1$ and $\widetilde{C}_2$ the proper transforms on the surface $\widetilde{S}_4$ of the conics $C_1$ and $C_2$, respectively.
By Lemma~\ref{lemma:point-not-on-exceptional}, we see that $Q\in\widetilde{C}_1\cup\widetilde{C_2}$.
If $Q\in\widetilde{C}_1$, then
$$
2-m=\widetilde{D}\cdot\widetilde{C}_1\geqslant\widetilde{m}
$$
which implies that $m+\widetilde{m}\leqslant 2$. On the other hand, we have
$m+\widetilde{m}>\frac{2}{\lambda}>\frac{8}{3}$ by \eqref{equation:m-m-big}.
Hence, we see that $Q\not\in\widetilde{C}_1$ and $Q\in\widetilde{C}_2$.

Write $D=aC_2+\Omega$, where $a$ is a non-negative rational number,
and $\Omega$ is an effective $\mathbb{Q}$-divisor whose support does not contain the conic $C_2$.
Put $n=\mult_{P}(\Omega)$. Then
$$
2-4a=\big(H-aC_2\big)\cdot C_1=\Omega\cdot C_2\geqslant n.
$$
This gives $n+4a\leqslant 2$. In particular, $a\leqslant\frac{1}{2}$.

Denote by $\widetilde{\Omega}$ the proper transform of the $\mathbb{Q}$-divisor $\Omega$ on the surface $\widetilde{S}_4$,
and put $\widetilde{n}=\mult_Q(\widetilde{\Omega})$.
Then $n\geqslant\widetilde{n}$ and
$$
2+2a-n=\big(H-aC_2\big)\cdot C_2-n=\Omega\cdot C_2-n=\widetilde{\Omega}\cdot\widetilde{C}_2\geqslant\widetilde{n}.
$$
Hence, we have $n+\widetilde{n}\leqslant 2+2a$.
Using this inequality together with~$n+4a\leqslant 2$, we see that
$$
\widetilde{n}\leqslant 2+2a-n\leqslant 2+\frac{1}{2}\big(2-n\big)-n,
$$
which implies that $\frac{3}{2}n+\widetilde{n}\leqslant 3$.
This together with the fact that $\widetilde{n}\leqslant n$ shows that $\widetilde{n}\leqslant\frac{6}{5}$.

Rewrite the log pair \eqref{equation:log-pull-back} as
$(\widetilde{S}_4, \lambda a\widetilde{C}_2+(\lambda n+\lambda a-1)E+\lambda\widetilde{\Omega})$.
Since $\widetilde{n}\leqslant\frac{6}{5}$, we see that $\lambda\widetilde{n}<1$.
Hence, we can apply Theorem~\ref{theorem:Trento} to the pair \eqref{equation:log-pull-back} at the point $Q$.
This gives us that either
$$
\lambda (2+2a-n)=\lambda\big(\Omega\cdot C_2-n\big)=\lambda\widetilde{\Omega}\cdot\widetilde{C}_2>2\big(1-(\lambda n+\lambda a-1)\big)
$$
or $\lambda n=\lambda\widetilde{\Omega}\cdot E>2(1-\lambda a)$ (or both). In the first case, we have
$$
4a+n>\frac{4}{\lambda}-2>\frac{16}{3}-2=\frac{8}{3},
$$
because $\lambda<\frac{3}{4}$. In the second case, we get $n+2a>\frac{2}{\lambda}>\frac{8}{3}$.
On the other hand, we already proved that $4a+n\leqslant 2$.
This gives us the desired contradiction.
\end{proof}

\begin{lemma}
\label{lemma:C7}
The case \eqref{quartic-m2} is impossible.
\end{lemma}

\begin{proof}
Suppose that we are in the case \eqref{quartic-m2}.
Then $\mult_P(T_P)=2$ and $T_P$ is an irreducible quartic curve with a singular point $P$ of multiplicity~$2$
We have the following picture:
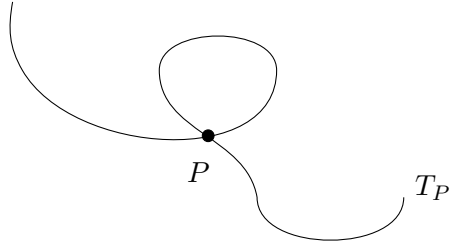
\begin{center}
\begin{tikzpicture}[scale=1.3]
\draw  (0.5,2) to [out=-100,in=120] (0.6,1.3) to [out=-60,in=-90]   (3.2,1.3) to [out=90,in=90] (2,1.3)  to [out=-90,in=100] (3,0)
to [out=-90, in=270] (4.5,0);
\node at (2.4,0.26) {$P$};
\node at (4.8,0.1) {$T_P$};
\draw [fill] (2.5,0.63) circle [radius=0.06];
\end{tikzpicture}
\end{center}

Since $T_P$ is irreducible, we have $C_\star=C$.
This gives $4=D\cdot C\geqslant 2m$, which implies that $m\leqslant 2$.
Thus, $Q\in\widetilde{T}_P$ by Lemmas~\ref{lemma:point-not-on-exceptional} and \ref{lemma:O-E-F}.
Therefore, we have
$$
4-2m=\widetilde{D}\cdot\widetilde{C}\geqslant\widetilde{m}
$$
which implies that $2m+\widetilde{m}\leqslant 4$. Using \eqref{equation:m-m-big}, we get
$4-m\geqslant m+\widetilde{m}>\frac{2}{\lambda}>\frac{8}{3}$,
which implies that $m\leqslant\frac{4}{3}$. But $m>\frac{4}{3}$ by \eqref{equation:m-big}.
\end{proof}

By Corollary~\ref{corollary:A} and Lemmas~\ref{lemma:B1}, \ref{lemma:B2}, \ref{conic+2lines}, \ref{cubic+line},
\ref{lemma:C1-C2}, \ref{lemma:C3}, \ref{lemma:C4}, \ref{lemma:C5}, \ref{lemma:C6}, and \ref{lemma:C7}, we obtain the desired contradiction.
This completes the proof of Theorem~\ref{theorem:quartic}.

\section{General surfaces of large degree}
\label{section:general}

In this section, we prove Theorem~\ref{theorem:general-surface}.
By Lemmas~\ref{lemma:3-4} and \ref{lemma:3-4-general}, it follows from

\begin{lemma}
\label{lemma:general-surface}
Let $S_d$ be a smooth surface in $\PP^3$ of degree $d$, and let $H$ be its hyperplane section.
Then  $\alpha(S_d,H)\leqslant\frac{2}{\sqrt{d}}$.
\end{lemma}

\begin{proof}
Let $P$ be a point in $S_d$, and let $f\colon\widetilde{S}_d\rightarrow S_d$ be the blow up of the surface $S_d$ at the point $P$.
Denote by $E$ the $f$-exceptional curve.
Fix \emph{any} positive rational number $m$ such that $m<\sqrt{d}$,
and take a positive integer $n$ such that $mn$ is an integer.
Then
$$
\big(f^*(nH)-nmE\big)^2=n^2\big(d-m^2\big)>0.
$$
This implies that the linear system $|f^*(nH)-nmE|$ is not empty for $n\gg 0$.
Indeed, we have
$$
h^2\Big(\widetilde{S}_4,\mathcal{O}_{\widetilde{S}_d}\big(f^*(nH)-nmE\big)\Big)=h^0\Big(\widetilde{S}_4,\mathcal{O}_{\widetilde{S}_d}\big(f^*((d-4-n)H)+(mn+1)E\big)\Big)=0
$$
for $n>d-4$ by Serre duality. Thus, if  $n$ is sufficiently big comparing to $d$, then
\begin{multline*}
h^0\Big(\widetilde{S}_d,\mathcal{O}_{S_d}\big(f^*(nH)-nmE\big)\Big)\geqslant\\
\geqslant\chi\big(\mathcal{O}_{\widetilde{S}_d}\big)+\frac{1}{2}\Big(\big(f^*(nH)-nmE\big)^2-\big(f^*(nH)-nmE\big)\cdot K_{\widetilde{S}_4}\Big)=\\
=\chi\big(\mathcal{O}_{\widetilde{S}_d}\big)+\frac{1}{2}\Big(n^2\big(d-m^2\big)-n(d-4)-nm\Big)>0
\end{multline*}
by the Riemann--Roch formula for surfaces.

Let us fix a positive integer $n$ such that $mn$ is an integer and $|f^*(nH)-nmE|$ is not empty.
Pick a divisor $\widetilde{M}$ in this linear system, so that $\widetilde{M}\sim n\widetilde{H}-nmE$.
Denote by $M$ the proper transform of the divisor $\widetilde{M}$ on the surface $S_d$.
Put $D=\frac{1}{n}M$.
Then $\mult_P(D)\geqslant m$, so that $\mathrm{lct}_P(S_d,D)\leqslant\frac{2}{m}$ by \eqref{equation:lct}.
This gives $\alpha(S_d, H)\leqslant\frac{2}{m}$, because $D\sim_{\mathbb{Q}} H$.
Since we can choose rational number $m<\sqrt{d}$ as close to $\sqrt{d}$ as we wish, we obtain $\alpha(S_d,H)\leqslant\frac{2}{\sqrt{d}}$.
\end{proof}

The idea of the proof of this lemma comes from \cite[Example~1.26]{Ch14}.

\begin{proof}[\bf Proof of Theorem~\ref{theorem:general-surface}] It follows from Lemma~\ref{lemma:3-4} and Lemma~\ref{lemma:3-4-general} that $\alpha_1(S_d,H)=\frac{3}{4}$ for a general surface $S_d$ in $\PP^3$. The claim follows from this fact together with Lemma~\ref{lemma:general-surface}.
\end{proof}

\section{Quintic, sextic and septic}
\label{section:small-degree}

Let $S_d$ be a surface in $\mathbb{P}^3$ that is given by
$$
\big(x^{d-2}+y^{d-2}+z^{d-2}+w^{d-2}\big)\big(xw+yz\big)+\big(y-z\big)^d-x^d=0,
$$
where $d\geqslant 2$. One can easily see that the surface $S_d$ is smooth.
Denote by $H$ its hyperplane section.
Arguing as in \cite[Example~3.9]{ChPaWo14}, we obtain

\begin{lemma}
\label{lemma:alpha-1-2}
Suppose that $d\leqslant 7$. Then $\alpha_1(X_d, H)>\frac{1}{2}$.
\end{lemma}

\begin{proof}
Let $C\subset\mathbb{P}^3$ be the curve defined by the intersection of the surface $S_d$ and the Hessian surface $\mathrm{Hess}(S_d)$ of $S_d$.
For the tangent hyperplane $T_P$ at a point $P\in S_d$, if the multiplicity of the curve $T_P\cap S_d$
at the point $P$ is at least $3$, then the curve $C$ is singular at the point $P$.
Using the computer algebra system \emph{Magma}, we checked that the curve $C$ is smooth.
Thus, the intersections of $S_d$ with its tangent planes do not have points of multiplicity $3$ or higher.
The later implies that $\alpha_1(S_d,H)>\frac{1}{2}$.
Indeed, each singular hyperplane section of $S_d$ is reduced by Lemma~\ref{lemma:Pukhlikov},
so that each its singular point is of type $\mathbb{A}_n$.
Then $\alpha_1(S_d,H)=\frac{1}{2}+\frac{1}{m}$,
where $m$ is the greatest integer such that a hyperplane section of $S_d$ has a singular point of type $\mathbb{A}_m$.
\end{proof}

On the other hand, we have

\begin{lemma}
\label{lemma:quintic-septic-sextic}
One has $\alpha_2(S_d,H)\leqslant\frac{3}{d}$.
\end{lemma}

\begin{proof}
We may assume that $d\geqslant 3$. Put $P=[0:0:0:1]$.
Let $M$ be the divisor that is cut out on $S_d$ by the equation $xw+yz=0$.
Locally at $P$, the divisor $M$ is given by $(y-z)^d=(-yz)^d=0$,
which implies that $\mathrm{lct}_P(S_4,M)=\frac{3}{2d}$.
Since $M\sim 2H$, we obtain $\alpha_2(S_d,H)\leqslant\frac{3}{d}$.
\end{proof}

\begin{corollary}
\label{corollary:septic-sextic}
If $d>5$, then $\alpha(S_d,H)<\alpha_1(S_d,H)$.
\end{corollary}

\begin{remark}\label{computer}We expect that  $\alpha(S_d,H)<\alpha_1(S_d,H)$ for $d=5$ as well.
By Lemma~\ref{lemma:alpha-1-2}, this claim follows from $\alpha_1(S_d,H)>\frac{3}{5}$.
To check the latter inequality one would have to find out if the intersections of $S_d$ with its tangent planes have a singularity of type $\mathbb{A}_9$ or worse.
This can be expressed as a system of polynomial equations in $4$ variables $x,y,z,w$:

Start with the equation of the quintic in variables $x,y,z,w$. Then intersect this with a symbolic plane $w=ax+by+cz$, by substitution. This gives a polynomial in $a,b,c,x,y,z$. Now we compute the discriminant of this equation with respect to $z$, which results in a huge polynomial in $a,b,c,x,y$. Let us denote this polynomial by $h$. If there is an $\mathbb{A}_9$ singularity, or worse, then the discriminant, as a polynomial in $x,y$ (when $a,b,c$ are treated as as parameters), should have a zero of multiplicity $10$ or higher. So the system of equations to consider consists of $h$ and all its derivatives of order up to $10$, as a system of polynomial equations in $a,b,c$, and $x$.

We used computer algebra to check whether or not this system has a solution,
but the computations did not finish after 1500 CPU seconds on a Pentium Pro with 2.7 GHz.
After reducing the system of equations modulo some small prime numbers (up to 293),
the program finished with the answer that the reduced system has no solution.
This can be interpreted as a strong evidence that $\alpha(S_d,H)<\alpha_1(S_d,H)$ for $d=5$.
\end{remark}


\begin{thebibliography}{10}

\bibitem{Vanya}
Ivan Cheltsov, \emph{Fano varieties with many selfmaps}, Adv. Math. \textbf{217} (2008), no.~1, 97--124.

\bibitem{Ch08}
Ivan Cheltsov, \emph{Log canonical thresholds of del {P}ezzo surfaces}, Geom. Funct. Anal. \textbf{18} (2008), no.~4, 1118--1144.

\bibitem{Ch13}
Ivan Cheltsov, \emph{Del {P}ezzo surfaces and local inequalities}, Automorphisms in birational and affine geometry, Springer Proc. Math. Stat., vol.~79, Springer, Cham, 2014, pp.~83--101.

\bibitem{Ch14}
Ivan Cheltsov, \emph{Worst singularities of plane curves of given degree}, to appear in J. Geom. Anal.

\bibitem{ChPaWo14}
Ivan Cheltsov, Jihun Park, and Joonyeong Won, \emph{Log canonical thresholds of certain {F}ano hypersurfaces}, Math. Z. \textbf{276} (2014), no.~1-2, 51--79.

\bibitem{ChSh}
Ivan Cheltsov, Constantin Shramov, \emph{Log canonical thresholds of smooth Fano threefolds},  Russian Math. Surveys \textbf{63} (2008), no. 5, 859--958

\bibitem{CDS}
Xiuxiong Chen, Simon Donaldson, and Song Sun, \emph{K\"ahler-{E}instein metrics   on {F}ano manifolds}, J. Amer. Math. Soc. \textbf{28} (2015), no.~1, 183-–278.

\bibitem{DK01}
Jean-Pierre Demailly and J{\'a}nos Koll{\'a}r, \emph{Semi-continuity of complex singularity exponents and {K}\"ahler-{E}instein metrics on {F}ano orbifolds}, Ann. Sci. \'Ecole Norm. Sup. (4) \textbf{34} (2001), no.~4, 525--556.

\bibitem{Eyss}
Philippe Eyssidieux, \emph{M\'etriques de {K}\"ahler-{E}instein sur les vari\'et\'es de {F}ano}, Ast\'erisque \textbf{380} (2016), 207--229.

\bibitem{CoKoSm}
J{\'a}nos Koll{\'a}r, Karen~E. Smith, and Alessio Corti, \emph{Rational and nearly rational varieties}, Cambridge Studies in Advanced Mathematics, vol.~92, Cambridge University Press, Cambridge, 2004.

\bibitem{Rubinstein}
Yanir Rubinstein, \emph{Smooth and singular {K}\"ahler-{E}instein metrics}, Contemp. Math. \textbf{630} (2014), 45--138.

\bibitem{Tian97}
Gang Tian, \emph{On {K}\"ahler-{E}instein metrics on certain {K}\"ahler manifolds with {$C_1(M)>0$}}, Invent. Math. \textbf{89} (1987), no.~2, 225--246.

\bibitem{Tian2012}
Gang Tian, \emph{Existence of {E}instein metrics on {F}ano manifolds}, Metric and differential geometry, Progr. Math., vol. 297, Birkh\"auser/Springer, Basel, 2012, pp.~119--159.

\bibitem{Tian-CPAM}
Gang Tian, \emph{$K$-stability and {K}\"ahler-{E}instein metrics}, Comm. Pure Appl. Math. \textbf{68} (2015), 1085--1156.


\end{thebibliography}
\end{document}